\documentclass[11pt]{amsart}
\usepackage{amsthm,amsfonts,amssymb,amsmath,oldgerm}
\numberwithin{equation}{section}
\numberwithin{figure}{section}
\numberwithin{table}{section}
\usepackage{fullpage}
\usepackage{amsmath}
\usepackage{setspace}
\usepackage{stmaryrd}
\usepackage{mathrsfs}
\usepackage{fancyhdr}
\usepackage{graphicx}
\usepackage{enumerate}
\usepackage{bm}
\usepackage{bbm}
\usepackage{dsfont}
\usepackage{multirow}
\usepackage{psfrag}
\usepackage[font=scriptsize]{caption}
\usepackage[font=scriptsize]{subcaption}
\usepackage{listings}
\usepackage{tikz}
\usepackage{tikz-cd}
\usepackage{empheq}
\usepackage{cases}
\usetikzlibrary{matrix} 

\usepackage{hyperref}
\hypersetup{
    colorlinks,               
    citecolor=blue,        
    filecolor=blue,         
    urlcolor=blue,         
    linkcolor=blue,
  }

\newcommand\dD{\mathrm{d}}
\newcommand\ds{\displaystyle}
\newcommand{\N}{{\mathbb N}}
\newcommand{\C}{{\mathbb C}}
\newcommand{\R}{{\mathbb R}}

\newcommand\cC{\mathcal C}
\newcommand\cD{\mathcal D}
\newcommand\cE{\mathcal E}
\newcommand\cI{\mathcal I}
\newcommand\cJ{\mathcal J}

\newcommand\cN{\mathcal N}
\newcommand\scR{R}
\newcommand\cP{\mathcal P}
\newcommand\cS{\mathcal S}
\newcommand\cT{\mathcal T}

\newcommand{\mC}{{\mathscr C}}
\newcommand{\mP}{{\mathscr P}}
\newcommand{\be}{\begin{equation}}
\newcommand{\ee}{\end{equation}}

\newtheorem{theorem}{Theorem}[section]
\newtheorem{proposition}[theorem]{Proposition}

\newtheorem{lemma}[theorem]{Lemma}


\begin{document}

\title[On the approximation of the von Neumann equation]{On the
  approximation of the von Neumann equation in the semiclassical
  limit. Part II : numerical analysis}

\author[Francis Filbet]{Francis Filbet}
\address[F.F.]{IMT, Université de Toulouse, 31062 Toulouse Cedex, France}
\email{francis.filbet@math.univ-toulouse.fr}

\author[François Golse]{Fran\c cois Golse}
\address[F.G.]{CMLS, \'Ecole polytechnique, 91128 Palaiseau Cedex, France}
\email{francois.golse@polytechnique.edu}

\begin{abstract}
  This paper is devoted to the numerical analysis of the Hermite spectral
  method proposed in \cite{FG1:24}, which provides, in the
  semiclassical limit,  an asymptotic
  preserving approximation of the von Neumann equation. More
  precisely, it relies on the use of so-called Weyl's variables to effectively address the
  stiffness associated to the equation. Then by employing a
  truncated Hermite expansion of the density operator, we successfully
  manage this stiffness and provide error estimates by leveraging the
  propagation of regularity in the exact solution.
  \end{abstract}

\date{\today}

\maketitle


\tableofcontents


\section{Introduction}
\label{sec:1}
\setcounter{equation}{0}
\setcounter{figure}{0}
\setcounter{table}{0}


Quantum dynamics exhibits high-frequency waves and requires the numerical resolution of small
wavelengths, which presents substantial computational difficulties
\cite{bader2014, fang2018, hochbruck1999,  jin2011}. Hence, for such
problems,  it is essential to consider an appropriate mesh strategy, which includes the correlation between time steps,
mesh size, and the physical wavelength $\lambda=2\pi\hbar/p$, where
$\hbar$ is the Planck
constant and $p$ the particle momentum \cite[section 17, Chapter III]{LL}. Similarly, numerical  schemes used to solve the
Schr{\"o}dinger equation typically demand that both the time step
$\Delta t$ and mesh size, in the semiclassical regime when
$\hbar\ll 1$, be of order $O(\hbar)$, or may even $o(\hbar)$ \cite{athanass, markowich1999}. Conversely,  time splitting methods have
the potential to enhance large time steps  when only the physical
observables are concerned, as  has been shown in
\cite{BaoJinMarko}. A fundamental concept for understanding these mesh strategies is the Wigner transform, which serves
as a useful tool for investigating the semiclassical limit of the
Schr{\"o}dinger equation.

Furthermore, in \cite{FGJinPaul},  uniform  error estimates of  time splitting methods have been established for the von Neumann equation
in the semiclassical regime.  This equation   describes the evolution of
mixed quantum states, and reduces to the Schr{\" o}dinger equation in the case of pure
quantum states. More precisely,  an operator $\hat\rho^\hbar$, called the density-matrix
operator,  satisfies the Liouville-von Neumann or simply von Neumann equation, in the operator formulation
\begin{equation}
\label{vN}
\left\{
\begin{array}{l}
\ds i\,\hbar \,\partial_t \hat\rho^\hbar \,=\, [H,\hat\rho^\hbar],
\\[1.1em]
\hat\rho^\hbar(0)\,=\,\hat\rho^\hbar_{\rm in}\,,
\end{array}\right.
\end{equation}
where $\hbar$ is the reduced Planck constant.  The operator $\hat\rho^\hbar$
is acting on the  Hilbert space $L^2(\R^d)
\, :=\, L^2(\R^d,\C)$ equipped with the inner product
\begin{equation}
  \label{inner:p}
\langle f, g\rangle \,:=\, \int_{\R^d} f(x)\,\overline{g(x)}\,\dD x\,,
\end{equation}
with associated norm given by $\|\cdot\|  :=
\sqrt{\langle\cdot,\cdot\rangle}$. Moreover, the quantum mechanical Hamiltonian $H$ is given, throughout this paper by
$$
H \,:=\, -\frac{\hbar^2}{2m}\,\Delta  \,+\, V (X),
$$
where $V$ a real-valued function such that $H$ is a self-adjoint operator on $L^2(\R^d)$. On the other hand, 
for all $\phi\in L^2(\R^d)$, we set $V \phi(X) \,:=\, V (X)\,\phi(X)$ while
$$
[H,\hat\rho^\hbar] \,:=\, H \,\hat\rho^\hbar - \hat\rho^\hbar\,H
$$
is the commutator. Here,  we suppose that  $\hat\rho^\hbar(t)$ is a density
matrix or operator on $L^2(\R^d)$, assuming that
\begin{equation*}
\hat\rho^\hbar(t)^*\,=\,\hat\rho^\hbar(t)\,\ge\,
0\,,\quad\text{Tr}(\hat\rho^\hbar(t))\,=\,1, \quad t \geq 0.
\end{equation*}
In particular, $\hat \rho^\hbar(t)$ is trace-class and therefore
Hilbert-Schmidt, hence  by \cite[Theorem 6.12]{brezis},  there exists a
unique function $\rho^\hbar(t|\,X,Y)\in {L^2_{X,Y}:= L^2(\R^d\times \R^d, \C)}$, such that  for each $\phi\in L^2(\R^d)$, we
define $\hat\rho^\hbar(t)\phi$ as the function
\begin{equation}
\label{naple:1}
X\,\mapsto \, \hat\rho^\hbar(t)\,\phi(X)\,:=\;\int_{\R^d}\rho^\hbar(t|\,X,Y)\,\phi(Y)\,\dD Y\,.
\end{equation}
Thus, using this latter representation, the von Neumann equation can be written as follows for the
function $(X,Y)\mapsto \rho^\hbar(t |\,X,Y)$ : 
\begin{equation}
  \label{vNn}
  \left\{
\begin{array}{l}
\ds i\,\hbar\,\,\partial_t \rho^\hbar(t|\,X,Y) \,=\,-\frac{\hbar^2}{2\,m}\,\left(\Delta_X\,-\,\Delta_Y\right)\,\rho^\hbar(t|\,X,Y)\,+\,\left(V(X)\,-\,V(Y)\right)\,\rho^\hbar(t|\,X,Y)\,,
\\[1.1em]
 \ds     \rho^\hbar(t=0|\,.,.)\,=\,\rho^\hbar_{\rm in}\,.
    \end{array}\right.
\end{equation}

First, it is important to emphasize that when $\hbar$ is significantly
smaller than $1$, the dynamics governed by the von Neumann equation
exhibits stiffness, resulting in a multiscale problem. Consequently,
one of the primary challenges in quantum dynamics is the development
of efficient numerical techniques capable of describing a wide
range of wavelengths. For example, we refer the reader to \cite{jin2022, miao2023, russo2013,russo2014}  and the
associated literature for discussions concerning the numerical
approximation of the Schr\"odinger equation within the semiclassical
limit.  Alternative methods based on stochastic algorithms \cite{XYZZ}
may be also applied in this context.

Furthermore, the solution of the von Neumann equation is a density
operator, {\it i.e.} a Hermitian, positive semi-definite operator with trace
equal to $1$. Preserving these intrinsic properties at the discrete
level is obviously crucial to ensure the generation of realistic
physical outcomes. In this regard, an initial approach was 
introduced in  \cite{hellsing1986efficient}  which involves the
successive application of short-time propagators evaluated using fast Fourier transforms. This method
exhibits considerable efficiency for short time intervals,
particularly when the density matrix is spatially localized. However, it is worth noting that it requires a small time step of the order
of $\hbar$ or potentially even smaller. Later in \cite{berman1992solution}, the authors proposed a numerical scheme
for a related model, again based on the Fourier pseudo-spectral
method. It allows a description both in configuration as well as in
momentum space. More recently, various structure preserving  schemes
have been studied to guarantee  trace conservation and positivity of
the discrete density matrix  allowing long time simulations.

Recently, we proposed a new approach based on the so-called Weyl's
variables introduced in part I \cite{FG1:24}  of the present paper  to address the
  stiffness associated with the equation, and applied a truncated
  Hermite expansion of the density operator : see \cite{FG1:24} for more details. Actually, this approach is inspired by the link between the von Neumann equation and the Wigner equation, which share the same kind of numerical difficulties arising from the nonlocal and highly-oscillating pseudo-differential
operator associated to the potential energy. In \cite{shao2011adaptive,xiong2016advective}, it has been shown that the Wigner function can be
discretized in a robust way using adaptive pseudo-spectral methods, where the oscillatory components introduced by the Wigner kernel are solved exactly. Finally, the oscillatory quantum effects can also
be mitigated by decomposing the potential into classical and quantum
parts \cite{sellier2015comparison, sels2013wigner} or by reformulating
the Wigner equation using spectral components of the force field
\cite{jcp2017}. Furthermore, as we will see, the use of Hermite
expansion allows for the partial preservation of the operator structure
$\hat\rho^\hbar$, which will be helpful in ensuring the stability of this approach. It is important to emphasize that in the Wigner equation, the term
involving the potential is nonlocal. However, when the equation is
formulated in Weyl's variables, this term remains local since it still
appears as a multiplication. Additionally, the
dependence on $\hbar$ is significantly reduced. After
Hermite discretization, our approach can be viewed as a form of
discretization of the Wigner equation, but this is specific to the
Hermite-Galerkin method since Hermite functions are eigenvalues of the
Fourier transform. Using another discretization method  would lead to
a different analysis of the Wigner equation and of the von Neumann
equation written in Weyl's variable (see \eqref{vNvW} below).

Our aim here is to analyze the Hermite spectral method proposed in
\cite{FG1:24}, for  equation \eqref{vNn}. We begin by rewriting the von Neumann
equation in terms of Weyl’s variables, which eliminates the stiffness
associated with $\hbar\ll 1$. The resulting equation is well-suited for the
semiclassical limit as $\hbar$ tends to zero (Theorem
\ref{th:01}). Next in Section \ref{sec:2}, we expand
the density matrix using Hermite polynomials, leading to a spectrally
accurate approximation. We provide new estimates on the solution of
\eqref{vNn}, demonstrating the propagation of regularity in a weighted
Sobolev space. Taking advantage of the regularity of the solution, we first derive error estimates
in terms of the number of Hermite modes, which are uniform with
respect to the parameter $\hbar$ (Theorem \ref{th:errorvN}). Additionally, we establish
error estimates for the Hermite spectral method applied to the
semiclassical limit model (Theorem \ref{th:errorSC}).  Then, Sections \ref{sec:3} and \ref{sec:4}  are
dedicated to the proofs of these results (Theorems \ref{th:errorSC}
and  \ref{th:errorvN}). Finally in Section \ref{sec:5}, we perform
numerical simulations to illustrate the spectral accuracy of the
Hermite method in the last section and refer the reader to \cite{FG1:24} for
more extensive numerical simulations.

\subsection{Quantum dynamics in  Weyl's variables}
\label{sec:1.1}
%
%

To remove the stiffness due to the parameter $\hbar$ in the von
Neumann equation, the key idea consists in rewriting
\eqref{vNn} in the new variables defined below,  called “Weyl's variables”:
\begin{equation}
  \label{WeylVar}
  \left\{
    \begin{array}{l}
      \ds x\,:=\,\frac{X\,+\,Y}{2}\\[0.9em]
      \ds y\,:=\,\frac{X\,-\,Y}{\hbar}\,,
    \end{array}\right.
\end{equation}
which can be explicitly inverted as
$$
X\,=\,x\,+\,\frac{\hbar}{2}\,y\,,\qquad Y\,=\,x\,-\,\frac{\hbar}{2}\,y\,.
$$
Although this terminology, introduced in \cite{FG1:24}, is not common, it is obviously suggested by
Weyl's  quantization (see formula (1.1.9) of
\cite{Lerner}). Hence, we set
\begin{equation}
\label{naple:2}
R^\hbar(t|\,x,y)\,:=\, \rho^\hbar\left(t |\,X,Y\right)\,,
\end{equation}
and find that $\rho^\hbar$ is solution to \eqref{vNn} if and only if
$R^\hbar$ satisfies 
\begin{equation}
  \label{vNvW}
   \left\{
     \begin{array}{l}
       
\ds \partial_t R^\hbar(t|\,x,y)\,=\,i\,\sum_{j=1}^d\partial_{x_j}\partial_{y_j}
R^\hbar(t|\,x,y)\,-\,i\frac{V\left(x+\frac\hbar 2\,
       y\right)\,-\,V\left(x-\frac\hbar 2\, y\right)}{\hbar}\,
       R^\hbar(t|\,x,y)\,,
\\[1.1em]
      R^\hbar(t=0 |\,.,.)\,=\,R^\hbar_{\rm in}\,.
      \end{array}\right.
  \end{equation}
  

{Before continuing with our presentation, let's summarize the
different formulations of the von Neumann equation \eqref{vN}. We introduced
\begin{itemize}
\item the density matrix
operator $\hat{\rho}^\hbar$, which refers to the solution of the von
Neumann equation \eqref{vN},
\item while $\rho^\hbar$ denotes the integral kernel
associated with $\hat{\rho}^\hbar$, which is a solution to the
equivalent equation \eqref{vNn},
\item the function $R^\hbar$ corresponds to  $\rho^\hbar$ expressed in terms of the Weyl variables \eqref{WeylVar}, and it
  is a solution of \eqref{vNvW}.
\end{itemize}
In the following, we will only consider the solution $R^\hbar$ to
\eqref{vNvW} and study its asymptotic behavior when $\hbar\ll 1$. }

Let us first review some elementary properties of the operator $\hat\rho^\hbar$ and
their interpretation in terms of the function $R^\hbar$. We first recall that $\hat\rho^\hbar$ is a density operator on $L^2(\R^{2d})$, that is
$\hat\rho^\hbar (t)^*\,=\,\hat\rho^\hbar (t)$, for all $t\in\R$, which means that  $R^\hbar$ satisfies 
\begin{equation}
  \label{SAvW}
R^\hbar(t |\,x,-y)\,=\,\overline{R^\hbar(t |\,x,y)}\,,
\end{equation}
whereas the trace property becomes 
\begin{equation}
  \label{TRvW}
\text{Tr}\left(\hat\rho^\hbar (t)\right) \,=\,\int_{\R^d}\rho^\hbar(t|\,X,X)\,\dD X\,=\,\int_{\R^d}
R^\hbar(t|\,x,0)\,\dD x\,=\,1\,.
\end{equation}
Translating the positivity condition on $\hat\rho^\hbar(t)\geq 0$ in terms
of the function $R^\hbar$ is more difficult : for all  $\phi \in L^2(\R^d)$, we have
\begin{eqnarray*}
0\,\leq \,\langle\phi|\,\hat\rho^\hbar(t)|\phi\rangle &:=&\iint_{\R^{2d}}\overline{\phi(X)}\, \rho^\hbar(t|\,X,Y)\,\phi(Y)\,\dD X\dD Y
\\[0.9em]
&=&\hbar^d\,\iint_{\R^{2d}} \overline{\phi(x+\tfrac\hbar 2  y)} \,R^\hbar(t|\,x,y)\,\phi(x-\tfrac\hbar 2 y)\,\dD x\,\dD y\,.
\end{eqnarray*}
We refer to \cite{GolseMoller} for a characterization of
rank-$1$ density operators in Weyl's variables.  Finally, for all
$t\geq 0$, we define  $\cS_\hbar(t)$ to be
the one parameter group associated to the von Neumann equation
\eqref{vNvW}, so that, for all $R^\hbar_{\rm in}\in L^2(\R^{2d})$,
\begin{equation}
\label{def:S}
R^\hbar(t) \,=\, \cS_\hbar(t)\,R^\hbar_{\rm in}\,.
\end{equation}
A simple change of variables implies the conservation of the $L^2$ norm for equation \eqref{vNvW}, summarized in the following proposition.

\begin{proposition}
  \label{prop:L2}
Consider $R^\hbar$ the solution to the von Neumann equation
\eqref{vNvW}.  Then, we have  for each $t\geq 0$,
\be
\label{esti:L2}
  \iint_{\R^{2d}}|R^\hbar(t |\,x,y)|^2\,\dD x\,\dD y\,=\,\| \cS_\hbar(t) R^\hbar_{\rm in}\|^2_{L^2}\,\,=\,\| R^\hbar_{\rm in}\|^2_{L^2}\,<\,\infty\,.
 \ee
\end{proposition}
\begin{proof}
  On the one hand, since  for all $t$, 
  $$
  \hat\rho^\hbar(t)\,=\,\exp(-itH/\hbar)\,\hat\rho^\hbar(0)\,\exp(itH/\hbar),
  $$
  one has
\[
\hat\rho^\hbar(t)^*=\hat\rho^\hbar(t) \quad\text{ since }\,\hat\rho^\hbar(0)^*=\hat\rho^\hbar(0),\quad\text{ and }\quad\mathrm{Tr}(\hat\rho^\hbar(t)^2)=\mathrm{Tr}(\hat\rho^\hbar(0)^2)<\infty\,.
\]
Since the Hilbert-Schmidt norm of an operator is the $L^2$ norm of its integral kernel, 
\[
\|\rho^\hbar(t|\cdot,\cdot)\|^2_{L^2(\mathbf R^d\times\mathbf R^d)}=\mathrm{Tr}(\hat\rho^\hbar(t)^2)=\mathrm{Tr}(\hat\rho^\hbar(0)^2)=\|\rho^\hbar(0|\cdot,\cdot)\|^2_{L^2(\mathbf R^d\times\mathbf R^d)}\,.
\]
The simple change of variables \eqref{WeylVar}, leads to
\[
\|\rho^\hbar(t|\cdot,\cdot)\|^2_{L^2(\R^d\times\R^d)}\,=\,\hbar^d\| R^\hbar(t|\cdot,\cdot)\|^2_{L^2(\R^d\times\R^d)}\;,
\]
which implies in turn the conservation of the $L^2$ norm for
equation \eqref{vNvW} : for each $t\in[0,+\infty)$,
$$
  \iint_{\R^{2d}}|R^\hbar(t|\,x,y)|^2\,\dD x\,\dD y\,=\,\| 
R^\hbar_{\rm in}\|^2_{L^2}\,<\,\infty\,.
$$
  \end{proof}

  This latter property will be crucial in the
stability analysis  of our numerical method but, it is not enough
in itself to prove convergence. We shall also need to establish some regularity estimates obtained by propagating moments
and derivatives in $L^2$ for the solution of the von Neumann equation
\eqref{vNvW}.

\subsection{Semiclassical limit}
\label{sec:1.2}
The formulation of the von Neumann
equation using Weyl's variables facilitates the semiclassical
limit. Passing to the limit as $\hbar\to 0$ in the latter equation  does not lead
to any difficulty related to the stiffness in $\hbar$. More precisely,  since $V\in
\mC^1(\R^d)$ we have :
$$
\frac{V(x+\frac\hbar 2\, y)-V(x-\frac\hbar 2\, y)}{\hbar}\to\nabla V(x)\cdot y\qquad\text{as }\hbar\to 0\,.
$$
Therefore, setting
\begin{equation}
  \label{def:Eh}
\cE^\hbar(x,y) \,:=\,  \frac{V(x+\frac\hbar 2\, y)-V(x-\frac\hbar 2
\,  y)}{\hbar}\,-\, \nabla V(x)\cdot y\,,
\end{equation}
the function  $R^\hbar$ now  solves
\begin{equation}
  \label{vNvW2}
   \left\{
    \begin{array}{l}
\ds\partial_t R^\hbar(t|\,x,y)\,=\,i\,\sum_{j=1}^d\partial_{x_j}\partial_{y_j}
R^\hbar(t|\,x,y)\,-\,i\left(\nabla V(x)\cdot y  \,+\,\cE^\hbar(x,y)\right)
\,R^\hbar(t|\,x,y)\,,
\\[1.1em]
      R^\hbar(t=0 |\,.,.)\,=\,R^\hbar_{\rm in}\,,
      \end{array}\right.
\end{equation}
where $\cE^\hbar$ goes to zero as $\hbar\rightarrow 0$. The error made
in approximating the solution $R^\hbar$ of \eqref{vNvW}, or
equivalently \eqref{vNvW2},  by the solution $\scR$ obtained after neglecting the source term $\cE^\hbar$,  corresponds to the
semiclassical limit, given by
\begin{equation}
  \label{eq:lim}
  \left\{
    \begin{array}{l}
\ds\partial_t\scR(t|\,x,y)\,=\,i\,\sum_{j=1}^d\partial_{x_j}\partial_{y_j}\scR(t|\,x,y)\,-\,i \,\nabla V(x)\cdot y  \,\scR(t|\,x,y)\,,
\\[1.1em]
      \scR(t=0 |\,.,.) \,=\,\scR_{\rm in}\,.
      \end{array}\right.
      \end{equation}
  This limit was made rigorous by P.-L. Lions and T. Paul
 \cite{LionsPaul}, and  P. Markowich and C. Ringhofer \cite{MN89,MR89}
for each smooth potential $V\in \mC^\infty$. We also refer to \cite[Section 1.4]{ref2},
where an asymptotic expansion of $\cE^\hbar$ in powers of $\hbar$ is
performed for smooth potentials  $V\in \mC^\infty$. In the sequel, we define
$\cS_0(t)$ to be  the one-parameter group  corresponding to the semiclassical limit equation
\eqref{eq:lim}, so that
\begin{equation}
\label{def:S0}
\scR(t) \,=\, \cS_0(t) \,\scR_{\rm in}.
\end{equation}

We now restrict ourselves to the one-dimensional case
$d=1$ and present an error estimate between the solutions
$R^\hbar(t|\,\cdot,\cdot)$ of the von Neumann equation in Weyl’s variables  \eqref{vNvW} and $\scR(t|\,\cdot,\cdot)$ of the semiclassical equation 
\eqref{eq:lim} in terms of $\hbar$.  The key point here is to establish some  regularity
estimates on the solution to \eqref{eq:lim}. To this effect, we
introduce the following quantity for any function $F$ and each integer
$m\geq 0$ :
\begin{equation}
  \label{def:Nm}
\cN_m(t)\,:=\, \sum_{a+b+\alpha+\beta\le m\atop a,b,\alpha,\beta\in\N}\|x^a\partial^b_xy^\alpha\partial^\beta_y F(t)\|_{L^2}\,.
\end{equation}

Let us fix an integer $m\geq 0$ and suppose that the potential $V$ is such that
\begin{equation}
  \label{hyp:V1}
  V\in \mC^{m+1}(\R) \quad{\rm with}\quad m\ge 1, \quad{\rm and }\quad
  V(x)\to+\infty \quad {\rm as}\quad |x|\to\infty\,,
\end{equation}
and for all $n=2,\ldots,m+1$, there exists a constant $C>0$, independent of $\hbar$, such that 
\begin{equation}
  \label{hyp:V2}
\| V^{(n)}\|_{L^\infty}  \,\leq\, C,{\quad {\rm with }\quad V^{(n)} \,:=\,
\frac{d^n V}{d x^n}}\,. 
\end{equation}
We prove the following result.

\begin{theorem}
  \label{th:01}
 Assume that $V$ satisfies \eqref{hyp:V1}-\eqref{hyp:V2} with $m=3$ while  for all integers  $a$, $b$,
$\alpha$ and $\beta$ such that  $a+b+\alpha+\beta\le 3$, we have
\begin{equation}
  \label{hyp:Rhat3}
  \left\{
    \begin{array}{l}
\ds\| x^a\partial^b_xy^\alpha\partial^\beta_y \scR_{\rm in} \|_{L^2}  \,<
\, \infty,
      \\
      \ds\|  R^\hbar_{\rm in} \|_{L^2}  \,<
      \, \infty.
      \end{array}\right.
      \ds \end{equation} 
Then  the solution $R^\hbar$ to  the von Neumann equation \eqref{vNvW} and
the solution $\scR$ to \eqref{eq:lim}  with the initial data
$R^\hbar_{\rm in}$ and $\scR_{\rm in}$ satisfy 
$$
\|R^\hbar(t)\,-\,\scR(t)\|_{L^2} \,\le\, \|R^\hbar_{\rm in}\,-\,\scR_{\rm in}\|_{L^2} \,+\,\cC(t)\,\hbar^2\,,
$$
where $\cC(t)$ is given by the following expression
$$
\cC(t)\,:=\,\frac1{24}\,\|V^{(3)}\|_{L^\infty}\,\frac{e^{C_3[V]\,t}\,-\,1}{C_3[V]}\,\cN_{3}(0)\,,
$$
with  $C_3[V]$ defined later in \eqref{def:Cm} and $\cN_3(0)$  given in
\eqref{def:Nm} applied to $\scR_{\rm in}$.
\end{theorem}
{Let us first comment the assumption of Theorem \ref{th:01}. It is  possible
to reformulate it in terms of the more commonly used Schatten-class
norms for density matrix operators since we have in dimension $d\geq 1$,
$$
\hbar^d \,\|R\|^2_{L^2_{x,y}} \,=\,\|\rho\|^2_{L^2_{X,Y}} \,=\,\|\hat\rho\|^2_2.
$$
Furthermore, we have
$$
\left\{
  \begin{array}{l}
    \ds\partial_xR \;=\, (\partial_X+\partial_Y)\rho,
    \\[0.8em]
    \ds\partial_y R \;=\, \frac12\;\hbar\;(\partial_X-\partial_Y)\rho, 
\end{array}\right.
  $$
  which are  the (integral)  kernels of the operators
  $$
\left\{
  \begin{array}{l}
    \ds [\partial_X,\hat\rho] \\[0.8em]
    \ds \frac12\,\hbar\; [\partial_X,\hat\rho]_+\,,
  \end{array}\right.
$$
where we have used the following notation for the anticommutator of two operators $A$ and $B$
$$
[A,B]_+ \, :=\,AB\,+\;BA\,.
$$
Finally,  we have
$$
\left\{
  \begin{array}{l}
    \ds xR \,=\;\frac12\;(X+Y)\,\rho,
    \\[0.8em]
    \ds yR\,=\;\frac1\hbar\;(X-Y)\;\rho,
     \end{array}\right.
$$
    which are  the (integral) kernels of the operators
    $$
    \left\{
  \begin{array}{l}
    \ds\frac12\;[X,\hat\rho]_+,
    \\[0.8em]
    \ds\frac1\hbar \;[X,\hat\rho].
     \end{array}\right.
$$}
The proof of this theorem is presented in Section
\ref{sec:3}. Actually, it is relatively elementary since it is a
consequence of a simple Taylor expansion of the potential $V$, but the propagation of regularity of the solution to the
semiclassical limit problem \eqref{eq:lim} is required. This will be achieved by
studying the functional $\cN_m(t)$, defined in \eqref{def:Nm}, applied to the solution $\scR(t)$
to \eqref{eq:lim}. This result can be interpreted as a modulated
energy estimate where only regularity on the solution of the limit
equation  is required (see for instance \cite{BGL} or \cite{B}).

\section{Hermite Spectral method}
\label{sec:2}
\setcounter{equation}{0}
\setcounter{figure}{0}
\setcounter{table}{0}


The purpose of this section is to present a formulation of the von
Neumann equation \eqref{vNvW} written in Weyl's variables based on Hermite
polynomials. We first use Hermite polynomials in the $y$
variable and write the von Neumann  equation  \eqref{vNvW}  as an infinite hyperbolic system for the Hermite coefficients
depending only on time and space. The idea is to apply a Galerkin
method keeping only a small finite set of orthogonal
polynomials rather than discretizing the operator $R^\hbar$ in $y$ on a grid. The merit of using
an orthogonal basis like the so-called scaled Hermite basis has been
shown in \cite{Holloway1996, Schumer1998} and more recently in
\cite{bessemoulin2022cv,ref:5, bf1-2024,bf2-2024,b-2024, Filbet2020} for the Vlasov-Poisson system.
We define the  weight function $\omega$ as 
\begin{equation*}
\omega(y):=\pi^{-1/4}\,e^{-y^2/2}\,,\qquad y\in\R\,,
\end{equation*}
and the sequence of Hermite polynomials (called the ``physicist's Hermite
polynomials'') by
\begin{equation*}
H_k(y)\,:=\,(-1)^k\,\omega^{-2}(y)\,\partial_y^k\omega^2(y)\,,\qquad k\ge 0\,.
\end{equation*}

In this context the family of Hermite functions  $\left(\Phi_k\right)_{k\in\N}$ defined as
\begin{equation*}
\Phi_k(y)\,:=\,\frac1{\sqrt{2^kk!}}\,\omega(y)\, H_k(y)\,,\qquad k\ge 0
\end{equation*}
is a complete,  orthonormal system for the classical $L^2$ inner product, that is,
\[
\int_{\R}\,
\Phi_{k}(y)\,\Phi_{l}(y)\,\dD y
\,=\,
\delta_{k,l}\,, \quad{\rm and}\quad \overline{{\rm Span}\{\Phi_k, \quad k\geq 0\}}= L^2(\R).
\]
The sequence $\ds
\left(
H_{k}
\right)_{k\in\N} 
$ defined in \eqref{def:Hk} satisfies the following recursion relation
:
\begin{equation}
  \label{def:Hk}
  \left\{
    \begin{array}{l}
\ds H_{-1}=0, \quad H_{0}=1,
      \\[0.9em]
 \ds H_{k+1}(y)\,=\, 2\,y\,H_k(y) \,-\,  2\,k\,H_{k-1}(y)\,,\quad\, k\,\geq\,0\,.
\end{array}
      \right.
      \end{equation}
We now expand the solutions to the von Neumann and 
semiclassical limit equations  in terms of Hermite functions.

\subsection{Hermite approximation of  the von Neumann equation}
\label{sec:2.1}
 
We consider the decomposition of $R^\hbar$, solution to 
 \eqref{vNvW}, into its components 
$
R^\hbar\,=\,(R^\hbar_{k})_{k\in\N}
$
in the Hermite basis :
\begin{equation}
  \label{def:Rh}
R^\hbar\left(t|\,x,y\right)
\,=\,
\sum_{k\in\N}\,
R^\hbar_{k}
\left(t|\,x
\right)\,\Phi_{k}(y)\,,
\end{equation}
where
\begin{equation}
\label{def:Rhk}
R^\hbar_k(t|\,x) \,:=\,\int_{\R}R^\hbar(t|\,x,y)\,\Phi_k(y)\,\dD y\,,\quad k\in\N\,.
\end{equation}
Since the Hermite functions verify the following relations \cite{Herm} :
\begin{equation} 
 \label{2-bis}
\Phi_k'(y)\,=\, -\sqrt{\frac{k+1}{2}}\,\Phi_{k+1}(y) \,+\, \sqrt{\frac{k}{2}}\,\Phi_{k-1}(y) \,,\quad\forall\, k\,\geq\,0\,,
\end{equation}
and
\begin{equation}
  \label{yPhi}
y\,\Phi_k(y)\,=\, \sqrt{\frac{k+1}{2}}\,\Phi_{k+1}(y) \,+\, \sqrt{\frac{k}{2}}\,\Phi_{k-1}(y) \,,\quad\forall\, k\,\geq\,0\,,
\end{equation}
we have
\begin{equation*}
\int_{\R}\partial_yR^\hbar(t|\,x,y)\,\Phi_k(y)\,\dD
  y \,=\,\ds -\sqrt{\frac{k}2}\,R^\hbar_{k-1}(t|\,x)\,+\,\sqrt{\frac{k+1}2}\,R^\hbar_{k+1}(t|\,x)\,,
\end{equation*}
whereas
\begin{eqnarray*}
\int_{\R}R^\hbar(t|\,x,y)\,y\,\Phi_k(y)\,\dD
  y&=&\ds \sqrt{\frac{k}2}\,
       R^\hbar_{k-1}(t|\,x)\,+\,\sqrt{\frac{k+1}2}\,R^\hbar_{k+1}(t|\,x)\,,
\end{eqnarray*}
with $R^\hbar _{-1}=0$. Furthermore, the contribution of the potential $V$
in \eqref{def:Eh}, now becomes
$$
\int_{\R}\left\{\frac{V(x+\frac\hbar 2\, y)-V(x-\frac\hbar 2\, y)}{\hbar}
  - V^\prime(x)\,y\right\}\,
R^\hbar(t|\,x,y)\Phi_k(y)\,\dD y\,=\,\sum_{l\in\N}\cE^\hbar_{k,l}(x)\,R^\hbar_l(t|\,x)\,,
$$
with
\begin{equation}
\cE^\hbar_{k,l}(x\,):=\,\int_{\R}\cE^\hbar(x,y)\, \Phi_k(y)\Phi_l(y)\,\dD y\,,\qquad k,\,l\,\in\,\N\,,
\label{def:Ekl}
\end{equation}
where the function $\cE^\hbar(.,.)$ is defined in
\eqref{def:Eh}. Before  going further, let us first review some
properties of $\cE_{k,l}^\hbar$. Obviously $\cE^\hbar_{k,l}(x)\in\R$ and we
also have
\begin{equation}
   \label{Ekl:sym}
\cE_{k,l}^\hbar(x) \,=\, \cE_{l,k}^\hbar(x)\,,\quad x\in\R, \,\,k,l\geq 0\,. 
\end{equation}
Furthermore, using that  $H_k$ is proportional to the even
function $\omega^{-2}$ times the $k$-th derivative  of the even
function $\omega^2$,   it shows that
\begin{equation}
  \label{am:0}
H_k(y)\,=\,(-1)^k\, H_k(-y)
\end{equation}
and from \eqref{def:Eh},  we also get that
$$
\cE^\hbar(x,y)=\cE^\hbar(x,-y).
$$
Therefore,  applying the change of  variable
$y\mapsto -y$ in \eqref{def:Ekl}, we have 
\begin{equation}
  \label{toto:1}
\cE^\hbar_{k,l}(x) \,=\, -(-1)^{k+l}  \,\cE^\hbar_{k,l}(x)\,,\quad
x\in\R, \quad k,l\geq 0\,. 
\end{equation}
This latter property implies that $\cE^\hbar_{k,l}$ is identically $0$
when ever $k+l$ is even. We thus end up with the formulation of the von Neumann equation
written in the Hermite basis for all integers $k\geq 0$, 
  \begin{equation}
\label{def:eqRhk}
  \left\{
    \begin{array}{l}
  \ds\partial_t  R^\hbar_{k} \,+\, i\,\left(\sqrt{\frac k 2}\,
      \cD\,R^\hbar_{k-1}\,+ \, \sqrt{\frac{k +1}{2}}\,
      \cD^\star R^\hbar_{k+1}\right)   \,=\,  -\,i\,\sum_{l\geq 0} \cE^\hbar_{k,l}\,R^\hbar_l\,,   
        \\[1.2em]
         \ds R^\hbar_{k}(t=0|\,.) =  R^\hbar_{{\rm in},k}\,,
\end{array}\right.
\end{equation}
where $\cD$ and $\cD^{\star}$ are given by
\begin{equation*}
\left\{
\begin{array}{ll}
\ds  \cD \,u   &:=\,\ds +\partial_{x}  u \,+\, V^\prime(x)\, u\,, \\[1.1em]
\ds\cD^\star \,u  &:=\,\ds -\partial_{x}  u \,+\, V^\prime(x)\, u\,.
\end{array}\right.
\end{equation*}
Let us emphasize an important property satisfied by $\cD$, which
we will need to recover later on, in the discrete setting. The
operator $\cD^\star$ is its adjoint operator in $L^2(\R)$, meaning that for all $u$, $v\in {\rm Dom}(\cD)$, it holds
 \be 
 \label{prop1:A}
\left\langle\cD^\star u ,\, v\right\rangle = \left\langle u,\,\cD v\right\rangle,
\ee
 where $\langle.,\,.\rangle$ denotes the classical inner product \eqref{inner:p} in
 $L^2(\R)$.
 Then we define the (closed) linear space $\mP_N(\R)\subset L^2(\R)$ of dimension $N+1$ :
\begin{equation}
  \label{def:PN}
\mP_{N}(\R) \,:=\, {\rm Span}\{ \Phi_k, \quad k=0,\ldots,N\}\,, 
\end{equation}
and construct an approximation $R^{\hbar,N}=(R^{\hbar,N}_k)_{0\leq k\leq N}
\in \mP_{N}(\R)$,  solution to the
following system obtained after neglecting
 Hermite modes of order larger than $N$, that is,  for $0\leq k \leq {N}$ 
 \begin{equation}
  \label{Hermite:DN}
  \left\{
    \begin{array}{l}
  \ds\partial_t  R^{\hbar,N}_{k} \,+\, i\,\left(\sqrt{\frac k 2}\,
      \cD R^{\hbar,N}_{k-1}\,+ \, \sqrt{\frac{k +1}{2}}\,
      \cD^\star R^{\hbar,N}_{k+1}\right)   \,=\,  -\,i\,\sum_{l= 0}^{N}\cE^\hbar_{k,l}\,R^{\hbar,N}_l\,,   
        \\[1.2em]
         \ds R^{\hbar,N}_{k}(t=0|\,.) =  R^{\hbar}_{{\rm in},k}\,,
\end{array}\right.
\end{equation}
where $R^{\hbar,N}_{-1} = R^{\hbar,N}_{N+1}=0$.  The right-hand side
of \eqref{Hermite:DN} will later be designated as  ''the nonlocal
term in \eqref{Hermite:DN}''.  (this may be slightly improper since the left-hand side also involves $R^{\hbar,N}_{k\pm 1}$ which can be viewed as nonlocal in terms of $(R^{\hbar,N}_k)_{0\le k\le N}$)

We are now ready to present our second main result  on error estimates of between the
Hermite-Galerkin approximation and the solution to the von Neumann equation
\eqref{vNvW}. The use of Weyl variables allows us to obtain uniform
accuracy with respect to the parameter $\hbar$, ensuring naturally the
asymptotic preserving property of the Hermite approximation for the
von Neumann equation. Again  the key point here is to establish some  regularity
estimates on the solution to the von Neumann equation  written in
Weyl's variables. 

\begin{theorem}
  \label{th:errorvN}
  Let $p\geq 1$ and assume that $V$ satisfies
  \eqref{hyp:V1}-\eqref{hyp:V2} with $m=2(p+1)$, while  for all integers  $a$, $b$,
  $\alpha$ and $\beta$ such that  $a+b+\alpha+\beta\le m$, 
\begin{equation}
\label{hyp:R}
\| x^a\partial^b_xy^\alpha\partial^\beta_y R^\hbar_{\rm in} \|_{L^2}  \,<
\, \infty.
\end{equation}
Then  the solutions $R^\hbar$
to \eqref{vNvW} and  $R^{\hbar,N}$ to \eqref{Hermite:DN} satisfy
\begin{eqnarray*}
\| R^{\hbar,N}(t) - R^{\hbar}(t) \|_{L^2} & \leq &   \frac{K_p}{(2{N}+3)^p} \,
                                             e^{C_{2p}[V]\,t}\,\cN_{2p}(0)
  \\
  &+& \frac{\sqrt{C_p\,\max(1,|V'(0)|,\|V''\|_{L^\infty})}}{(2{N}+3)^{p-1/2}}\,\frac{e^{C_{2p+1}[V]t}-1}{C_{2p+1}[V]}
      \, \cN_{2p+1}(0)
    \\
                          &+&
      \frac{K_{p+1}\,\hbar^2}{4\,(2N+3)^{p-1/2}}\,\|V^{(3)}\|_{L^\infty}\, \,\frac{e^{C_{2p+2}[V]t}-1}{C_{2p+2}[V]}\, \cN_{2p+2}(0)\,,
\end{eqnarray*}
where $\cN_m(t)$ is defined in \eqref{def:Nm} and applied to
$R^\hbar$, whereas  $C_p$, $K_p$ and $K_{p+1}$ are constants depending only on $p$, while $C_{l}[V]$, for
$l=2p,\ldots, 2p+2$, is the constant defined later  in
Proposition \ref{prop:reg_R}.
\end{theorem}

This result is particularly remarkable because, on the one hand, it
provides spectral accuracy, that is, the convergence rate in $N$
depends only on the regularity of the solution to the continuous
problem \eqref{vNvW}. On the other hand, these estimates are uniform
with respect to the parameter $\hbar \ll 1$, which guarantees a
uniform accuracy of the numerical approximation for different physical
regimes. Thus we can guarantee the asymptotic preservation property in the limit where $\hbar$ tends to $0$.

The proof is split in two steps. First, we  establish  stability
estimates  for the $L^2$ norm of $R^{\hbar,N}$. In other words we obtain a
uniform in $\hbar$ error estimate for the approximation of the
solution of the von Neumann equation (in Weyl's variables) by the
Hermite Galerkin system \eqref{Hermite:DN}. Then we prove the propagation of regularity of
the solution to the von Neumann equation \eqref{vNvW}  written in
Weyl's variables.

 \subsection{Hermite approximation of  the semiclassical limit   equation}
 \label{sec:2.2}
 
 As we have seen before, the formulation \eqref{Hermite:DN} is well adapted to the semiclassical limit
as $\hbar\rightarrow 0 .$ Indeed, the Hermite  formulation of the
semiclassical limit equation \eqref{eq:lim} is simply obtained by neglecting the right-hand side term in
\eqref{Hermite:DN}. Thus,  we construct an approximation
$\scR^N=(\scR_k^N)_{0\leq k\leq  N}$, which satisfies, for all $k\in\{0,\ldots,N\}$,
\begin{equation}
  \label{Hermite:limN}
    \left\{
    \begin{array}{l}
  \ds\partial_t  \scR_{k}^N \,+\, i\,\left(\sqrt{\frac k 2}\,
      \cD\scR_{k-1}^N\,+\, \sqrt{\frac{k +1}{2}}\,
      \cD^\star \scR_{k+1}^N\right)   \,=\,  0\,,   
        \\[1.2em]
      \ds \scR_{k}^N(t=0 |\,.) \,=\,  \scR_{{\rm in},k}\,.
      
\end{array}\right.
\end{equation}

Following the same strategy as the one applied for the von Neumann
equation, we provide error estimates of the Hermite-Galerkin
method  for smooth
solutions of the semiclassical limit equation \eqref{eq:lim}. Hence, using the propagation of regularity, we prove the following result.

\begin{theorem}
  \label{th:errorSC}
  Let $p\geq 1$ and assume that  $V$ satisfies
  \eqref{hyp:V1}-\eqref{hyp:V2} with $m=2p+1$ while for all integers  $a$, $b$,
$\alpha$ and $\beta$ such that  $a+b+\alpha+\beta\le m$,  the initial data $\scR_{\rm in}$ satisfies
\begin{equation}
\label{hyp:Rhat}
\| x^a\partial^b_xy^\alpha\partial^\beta_y \scR_{\rm in}\|_{L^2}  \,<
\, \infty.
\end{equation}
Then  the solutions $\scR$
to \eqref{eq:lim} and  $\scR^{N}$ to \eqref{Hermite:limN} satisfy
\begin{eqnarray*}
&& \|\scR(t)\,-\,\scR^{N}(t)\|_{L^2} \,\leq\, \frac{K_p\,e^{C_{2p}[V]t}}{(2{N}+3)^p}\,\cN_{2p}(0)
\\[0.9em]
                                      &&+
                                          \frac{\sqrt{C_p\max(1,|V'(0)|,\|V''\|_{L^\infty})}}{(2{N}+3)^{p-1/2}}\,\frac{e^{C_{2p+1}[V]t}-1}{C_{2p+1}[V]}\,\cN_{2p+1}(0)\,,
\end{eqnarray*}
where $\cN_m(t)$ is defined in \eqref{def:Nm} and applied to
$\scR$, whereas  $C_p$ and $K_p$ are constants depending only on $p$, while the constants
$C_{l}[V]$, for $l=2p$ and $2p+1$ are given later in \eqref{def:Cm}.
\end{theorem}

Let us emphasize that this result requires slightly less
regularity on the initial data than the one  in Theorem
\ref{th:errorvN}. This is due to the absence of non-local in $k$ terms
 in the right-hand side of \eqref{Hermite:limN}, at variance with
 \eqref{Hermite:DN}.


{In the following diagram, we present a synthetic overview of the
 preceding results concerning the semiclassical limit, as well as the
 convergence results of the spectral method based on the Hermite
 expansion. It is important to note here that the use of Weyl
 variables allows us to commute the limits $N$ tends to infinity and
 $\hbar$ tends to $0$ (Theorem \ref{th:errorvN}). }

 \begin{center}
\begin{tikzpicture}
	\matrix(m)[matrix of math nodes, row sep=7em, column sep=10em,text height=1.5ex, text depth=0.7ex]
	{
	 R^{\hbar}  \textrm{ solution to \eqref{vNvW}}          &  R   \textrm{ solution to \eqref{eq:lim}}
        \\
	R^{\hbar,N} \textrm{ solution to \eqref{def:eqRhk}}  & R^{N} \textrm{ solution to \eqref{Hermite:limN}} \\
	};
	\path[->]
	(m-1-1) edge node[above] { $\hbar\rightarrow 0$ (Theorem \ref{th:01}) } (m-1-2);
 	\path[->]
	(m-2-1) edge node[left,align=right]{$\,N\rightarrow+\infty$ and
          \\ uniformly in $\hbar>0$ \\  \textrm{(Theorem \ref{th:errorvN})} } (m-1-1);
	\path[->]
	(m-2-2) edge node[right, align=left] {$\,N\rightarrow+\infty$ \\  \textrm{(Theorem \ref{th:errorSC})}  } (m-1-2);
\end{tikzpicture}
\end{center}

 In order to offer a smooth presentation, we will
first present the proof of  Theorem \ref{th:errorSC}) in Section
\ref{sec:3} since it requires less tedious calculations
and then we will focus on the estimate of the  non-local term in the
right-hand side of \eqref{Hermite:DN} to prove Theorem
\ref{th:errorvN} in  Section \ref{sec:4}.

It is worth mentioning that the proof of these three results follows
the lines of the famous Lax theorem in numerical analysis
\cite{lax1, lax2}, since we make use of the regularity of the solution to
get an error estimate. A similar approach, either called the ``relative entropy
method'' or the ``modulated energy method'',  has also been widely used in
kinetic theory to compare the solutions of a kinetic and fluid
equations \cite{BGL} or \cite{B}.  In this work, we treat
simultaneously the error estimates with respect to the physical
parameter $\hbar$ and the discrete parameter $N$. 

{It is important to emphasize that the use of a Hermite expansion is
not essential to obtain uniform estimates with respect to
$\hbar$ leading to to the asymptotic preserving  property, since
these estimates follow from the reformulation in Weyl's
variables. Conversely, employing Hermite functions enables us to
achieve spectral accuracy, for which the order of accuracy only
depends on the regularity of the solution and is not fixed {\it a priori}
by the discretization method. Therefore, other types of numerical
approximations (such as Fourier methods \cite{berman1992solution, xiong2016advective}, finite differences, etc.)
could also be considered.}

\subsection{Preliminary results}
\label{sec:2.3}
Before  starting our analysis on the discretization of the von Neumann
equation using Hermite functions, let us briefly recall some basic
properties on spectral methods.

Consider  $G \in L^2(\R)$ written  in the Hermite basis
\begin{equation*}
G(y)
\,=\,
\sum_{k\in\N}\,
 G_{k}\,\Phi_{k}(y)\,,
\end{equation*}
where
$$
G_k \,:=\,\int_{\R} G(y)\,\Phi_k(y)\,\dD y\,,\quad k\ge 0\,.
$$
The Hermite-Galerkin method consists in truncating the series, keeping 
only  the $N$ first coefficients $G_k$ in the expansion
of the  function $G$ on the  basis $(\Phi_k)_{k\in\N}$ of Hermite  functions.
Thus,  let $\cP_{N}$ be the orthogonal projection from $L^2(\R)$ to
$\mP_{N}(\R)$, where $\mP_N(\R)$ is defined in \eqref{def:PN}. Let us first evaluate the ''error'' of the projection operator $\cP_{N}:
L^2(\R)\to \mP_{N}(\R)$ when applied to smooth functions.
\begin{lemma}
  \label{lem:4}
  Let  $G\in L^2(\R)$ satisfy, for some integer $p\ge 0$,
  $$
\|y^{2p}\, G\|_{L^2} + \| G\|_{H^{2p}} \,< \,\infty.
  $$
Then, we have
$$
\|G-\cP_{N}G
\|_{L^2}\,\le\,\frac1{(2N+3)^{p}}\left(\int_{\R}\left|\left(y^2-\frac{\dD^2}{\dD y^2}\right)^p\,G(y)\right|^2 \dD y\right)^{1/2}\,.
$$
\end{lemma}
\begin{proof}
On the one hand, since the Hermite functions is a complete family of orthonormal
functions in $L^2(\R)$, we deduce from the Parseval equality that
$$
\|G-\cP_{N}G \|^2_{L^2}\,=\,\sum_{k>{N}}|G_k|^2\,.
$$
On the other hand, for each integer $p\geq 0$, we  apply again the  Parseval equality
to find that
$$
\int_\R\left|\left(y^2-\frac{\dD^2}{\dD y^2}\right)^p \,G \right|^2\dD
y\,=\,\sum_{k\ge 0}\left|\left(\left(y^2-\frac{\dD^2}{\dD y^2}\right)^p\,G\right)_k\right|^2
$$
and observe that 
$$
\frac12\,\left(y^2-1-\frac{\dD^2 }{\dD y^2}\right)\,\Phi_k(y)\,=\,k\,\Phi_k(y)\,,\quad k\ge 0\,.
$$
Hence,
$$
\left(y^2-\frac{\dD^2 }{\dD y^2}\right)^p\,\Phi_k(y)\,=\,(2k+1)^p\,\Phi_k(y)\,,\quad k\ge 0\,,
$$
and for all integer $p\geq 0$, we find that 
$$
\left(\left(y^2-\frac{\dD^2 }{\dD y^2}\right)^p G\right)_k
\,=\,\int_\R G(y)\,\left(y^2-\frac{\dD^2}{\dD y^2}\right)^p\Phi_k(y)\dD y\,=\,(2k+1)^p\,G_k\,.
$$
Therefore,
$$
\int_{\R}\left|\left(y^2-\frac{\dD^2}{\dD y^2}\right)^p G(y)\right|^2\dD
y\,=\,\sum_{k\ge 0}(2k+1)^{2p}\,|G_k|^2\,.
$$
This leads to
\begin{eqnarray*}
\|G\,-\,\cP_{N} G \|^2_{L^2}&=&\sum_{k>N}|G_k|^2
  \\
  &\leq&  \frac1{(2{N}+3)^{2p}}\,\sum_{k>{N}}(2k+1)^{2p}|G_k|^2
\\
&\le& \frac1{(2{N}+3)^{2p}}\,\int_{\R}\left|\left(y^2-\frac{\dD^2}{\dD
      y^2}\right)^p G(y)\right|^2\dD y\,,
\end{eqnarray*}
which is the sought inequality.
\end{proof}

For sake of simplicity we first present our analysis on the
semiclassical model, where the non-local term is neglected. Then in
the second part, we treat the von Neumann equation and focus on the
contribution of the non-local term $\cE^\hbar$.

\section{Analysis of the semiclassical model}
\label{sec:3}
\setcounter{equation}{0}
\setcounter{figure}{0}
\setcounter{table}{0}

This section is mainly devoted to the propagation of regularity
on the solution to the semiclassical limit equation \eqref{eq:lim}.  We first study  the semiclassical
limit equation  in order to present the basic ideas. Then,  we provide
an error estimate with respect to $\hbar$ between $R^\hbar$ and $\scR$
when the solution to the semiclassical limit equation $\scR$ is
regular enough. Finally, we give an error estimate between the solution
$\scR$ to \eqref{eq:lim} and its approximation $R^N$ using the
Hermite-Galerkin method \eqref{Hermite:limN}.

\subsection{Propagation of regularity for the semiclassical limit equation}
\label{sec:3.1}


In this section we consider $\scR$, the solution to the
semiclassical limit equation \eqref{eq:lim}. We first recall  the
conservation of $L^2$ norm, and then prove the propagation of
regularity leading to the following result. Observe that our
assumptions on the smooth potential $V$ include the harmonic oscillator for which $V(x)=x^2$.

\begin{proposition}
  \label{prop:reg_Rhat}
  Let $m\geq 0$  and let $V$ be a potential satisfying
  \eqref{hyp:V1}, \eqref{hyp:V2} , and assume that, for each integer
  $a$, $b$, $\alpha$ and $\beta$ such that  $a+b+\alpha+\beta\le m$,
  the condition \eqref{hyp:Rhat}  on the initial data $\scR_{\rm in}$ is satisfied. Then the solution
$\scR$ to \eqref{eq:lim}  with $d=1$ satisfies,  for all  $t\ge 0$,
\begin{equation*}
\cN_m(t)
\,\le\, \cN_m(0)  \, e^{C_m[V]\,t}\,,
\end{equation*}
where $\cN_m$ is defined in  \eqref{def:Nm}  and $C_m[V]>0$ is given by
\begin{equation}
  \label{def:Cm}
C_m[V]\,:=\,\left(m^2+2m+m\|V''\|_{L^\infty}\,+\,m\,|V'(0)|\right)\,+\,2^m\,(1+m)\,m\,\max_{2\le n\le m+1}\|V^{(n)}\|_{L^\infty}\,.
\end{equation}
\end{proposition}
\begin{proof}
  First, observing that for all $t\geq 0$
  \begin{equation}
  \label{ref:31}
\| \scR(t) \|_{L^2}^2 \,=\, \int_{\R^2} |\scR (t|\,x,y)|^2 \dD
x\,\dD y \,=\, \| \cS_0(t)\,\scR_{\rm in} \|_{L^2}^2 \,=\, \| \scR_{\rm in} \|_{L^2}^2,
  \end{equation}
where $\cS_0(t)$ represents the one-parameter group \eqref{def:S0},  corresponding to the semiclassical limit equation
\eqref{eq:lim}, for which the $L^2$ norm is preserved by the evolution.
\\
Now, for each integer $a$, $b$, $\alpha$ and $\beta\geq 0$,
applying $x^a\,\partial_x^by^\alpha\partial_y^\beta$ to both sides of  \eqref{eq:lim}, we have
\begin{equation}
  \label{my:num1}
\begin{aligned}
\ds \partial_t\left(x^a\,\partial_x^by^\alpha\partial_y^\beta\scR\right) & \,=\, \left(i\,\partial_x\partial_y-i\,V'(x)y\right)\,\left(x^a\partial_x^by^\alpha\partial_\xi^\beta\scR\right)
\\
& \,+\,i\,\left[x^a\partial_x^by^\alpha\partial_y^\beta,\,\partial_x\partial_y-V'(x)\,y\right]\,\scR\,.
\end{aligned}
\end{equation}
On the one hand, computing the last term of the right-hand side in the
former equality and using that  $\left[z^\gamma,\,\partial_z\right]=
-\gamma\,z^{\gamma-1}$ for $z=x$ (resp. $y$) and $\gamma=\alpha$
(resp. $\beta$), we find
\begin{eqnarray*}
\left[ x^a\partial_x^by^\alpha\partial_y^\beta,\,\partial_x\partial_y-V'(x)\,y\right]
&=& \left[x^a\partial_x^by^\alpha\partial_y^\beta,\,\partial_x\partial_y\right]\,-\,\left[x^a\partial_x^by^\alpha\partial_y^\beta,\,V'(x)y\right]
\\
&=&\left [x^a,\,\partial_x\right]\,\partial_x^by^\alpha\partial_y^{\beta+1}\,+\,\partial_xx^a\partial_x^b\,\left[y^\alpha,\,\partial_y\right]\,\partial_y^\beta
\\
&-&x^a\,\left[\partial_x^b,\,V'(x)\right]\,y^\alpha\partial_y^\beta y\,-\,V'(x)\,x^a\partial_x^b y^\alpha\,\left[\partial_y^\beta,\,y\right].
\end{eqnarray*}
Then, we have 
$$
\begin{aligned}
\left[ x^a\partial_x^by^\alpha\partial_y^\beta,\,\partial_x\partial_y-V'(x)\,y\right]
\,=\, &- a\,x^{a-1}\,\partial_x^by^\alpha\partial_y^{\beta+1}\,-\,\alpha\,\partial_xx^a\partial_x^by^{\alpha-1}\,\partial_y^\beta
\\
&\,-\,\beta \,V'(x)\,x^a\,\partial_x^b y^\alpha\,\partial_y^{\beta-1}\,-\,x^a\left[\partial_x^b\,,V'(x)\right]\,y^\alpha\,\partial_y^\beta y
\\
\,=\,&-a\,x^{a-1}\,\partial_x^by^\alpha\partial_y^{\beta+1}\,-\,\alpha\, x^a\,\partial_x^{b+1}\,y^{\alpha-1}\partial_y^\beta
\\
&\,-\,\alpha \,a\,x^{a-1}\,\partial_x^by^{\alpha-1}\partial_y^\beta\,-\,\beta\, V'(x)\,x^a\,\partial_x^b \,y^\alpha\,\partial_y^{\beta-1}
\\
&\,-\,x^a\,\left[\partial_x^b,\,V'(x)\right]\,y^{\alpha+1}\,\partial_y^\beta\,-\,\beta\, x^a\,\left[\partial_x^b,\,V'(x)\right]\,y^\alpha\,\partial_y^{\beta-1}\,.
\end{aligned}
$$
Finally, applying the Leibniz formula
$$
\left[\partial_x^b,\,V'(x)\right]\,=\,\sum_{j=1}^b{b\choose j}\,V^{(j+1)}(x)\,\partial_x^{b-j}\,,
$$
we get
\begin{equation}
\begin{aligned}
-\left[x^a\,\partial_x^by^\alpha\partial_y^\beta,\,\partial_x\partial_y-V'(x)\,y\right] \,=\,& a\,x^{a-1}\,\partial_x^by^\alpha\partial_y^{\beta+1}\,+\,\alpha \,x^a\,\partial_x^{b+1}y^{\alpha-1}\partial_y^\beta
\\
& \,+\,\alpha \,a\,x^{a-1}\,\partial_x^by^{\alpha-1}\partial_y^\beta\,+\,\beta\, V'(x)\,x^a\,\partial_x^b y^\alpha\partial_y^{\beta-1}
\\
& \,+\,\sum_{j=1}^b{b\choose j}\,x^a\,V^{(j+1)}(x)\,\partial_x^{b-j}(y^{\alpha+1}\partial_y^\beta\,+\,\beta y^\alpha\partial_y^{\beta-1})\,.
\end{aligned}
\label{my:num2}
\end{equation}
On the other hand, we deduce from \eqref{my:num1} and the
Duhamel formula that 
$$
\begin{aligned}
x^a\,\partial_x^by^\alpha\partial_y^\beta\scR(t|\,x,y) &\,=\, \cS_0(t)(x^a\partial_x^by^\alpha\partial_y^\beta\scR_{\rm in})
\\
&\,+\,i\,\int_0^t \cS_0(t-s)\,\left[x^a\,\partial_x^by^\alpha\partial_y^\beta,\,\partial_x\partial_y-V'(x)y\right]\,\scR(s|\,x,y)\,\dD s\,,
\end{aligned}
$$
where $\cS_0(t)$ again denotes the one parameter group \eqref{def:S0}.  Hence, using  the conservation of the $L^2$ norm \eqref{ref:31} by $\cS_0(t)$,  we obtain that 
$$
\begin{aligned}
\left\|x^a\,\partial_x^b y^\alpha\partial_y^\beta\scR(t)\right\|_{L^2}
\, & \le\,\left\|x^a\partial_x^by^\alpha\partial_y^\beta\scR_{\rm in}\right\|_{L^2}
\\
\,&+\,\int_0^t\left\|\left[x^a\partial_x^by^\alpha\partial_y^\beta,\,\partial_x\partial_y\,-\,V'(x)\,y\right]\scR(s)\right\|_{L^2}\dD s\,.
\end{aligned}
$$
Thus, the former inequality and the computation of
$[x^a\partial_x^by^\alpha\partial_y^\beta,\partial_x\partial_y-V'(x)y]$
in \eqref{my:num2} imply that
\begin{eqnarray*}
\ds\|x^a\partial_x^by^\alpha\partial_y^\beta\scR(t)\|_{L^2}  & \le&  \ds\|x^a\partial_x^by^\alpha\partial_y^\beta\scR_{\rm in}\|_{L^2}
\\
& +&\int_0^t\left\{ a\,\|x^{a-1}\partial_x^by^\alpha\partial_y^{\beta+1}\scR\|_{L^2} \,+\,\alpha\,
   \|x^a\partial_x^{b+1}y^{\alpha-1}\partial_y^\beta
   \scR\|_{L^2} \right\}(s)\,\dD s
   \\
   & +&\int_0^t \left\{ \alpha \,a\,
        \|x^{a-1}\partial_x^by^{\alpha-1}\partial_y^\beta\scR\|_{L^2}   \,+\,\beta\,|V'(0)| \, \|x^{a}\partial_x^by^\alpha\partial_y^{\beta-1}\scR\|_{L^2} \right\}(s)\,\dD s
  \\
  &+& \beta\,\|V''\|_{L^\infty} \, \int_0^t \left\{ \|x^{a+1}\partial_x^by^\alpha\partial_y^{\beta-1}\scR\|_{L^2} \right\}(s)\,\dD s
  \\
& +&\max_{2\le n\le
  b+1}\|V^{(n)}\|_{L^\infty}\,\sum_{j=1}^b{b\choose j}
 \int_0^t \|x^a\partial_x^{b-j}y^{\alpha+1}\partial_y^\beta\scR\|_{L^2}(s) \,\dD s
     \\
& +& \beta\,\max_{2\le n\le
     b+1}\|V^{(n)}\|_{L^\infty}\,\sum_{j=1}^b{b\choose j}  \int_0^t \|x^a\partial_x^{b-j}y^\alpha\partial_y^{\beta-1}\scR\|_{L^2}(s) \,\dD s\,,
\end{eqnarray*}
where we used the mean-value inequality to prove that, for all $x\in\R$,
\begin{equation}
  \label{my:num3}
|V^\prime(x)|  \leq  |V(0)| \,+\, |x| \|V''\|_{L^\infty}.   
\end{equation}
Then we set 
$$
N_j(t) \,:=\,\sum_{a,b\geq 0}\sum_{\alpha,\beta\ge 0}
\|x^a\partial_x^by^\alpha\partial_y^\beta\scR(t)\|_{L^2} \, \delta_{a+b+\alpha+\beta,j}\,,
$$
where $\delta_{k,j}$ denotes the Kronecker symbol, which yields for
$j\geq 1$,
$$
\begin{aligned}
N_j(t)\,\le\, &
N_j(0)\,+\,\left(j^2+2j+j\|V''\|_{L^\infty}\right)\,\int_0^tN_j(s)\,\dD s
\,+\,j\,|V'(0)|\,\int_0^tN_{j-1}(s)\,\dD
s
\\
\,+\, & 2^j\,(1+j)\, \max_{2\le n\le j+1}\|V^{(n)}\|_{L^\infty}\,\sum_{k=1}^j\int_0^tN_{j+1-k}(s)\,\dD s\,.
\end{aligned}
$$
Thus, observing that 
$$
\cN_m(t)\,=\,\sum_{j=0}^mN_j(t)\,,
$$
we deduce the following inequality on $\cN_m(t)$ :
$$
\begin{aligned}
\cN_m(t) & \, \le\, \cN_m(0)\,+\,\left(m^2\,+\,2m\,+\,m\|V''\|_{L^\infty}\,+\,m\,|V'(0)|\right)\,\int_0^t\cN_m(s)\,\dD
s
\\
&\,+\,2^m\,(1+m)\,m\, \max_{2\le n\le m+1}\|V^{(n)}\|_{L^\infty} \,\int_0^t\cN_m(s)\,\dD s\,.
\end{aligned}
$$
Then we obtain the expected  upper bound by applying Gronwall's inequality. 
\end{proof}

This latter result is particularly important since it  allows us to
establish  several  error estimates. On the one hand, we prove an
error estimate with respect to $\hbar$ between
the solution $R^\hbar$ of the von Neumann equation and its semiclassical
limit $\scR$  (Theorem \ref{th:01}). On the other hand, we provide
an estimate with respect to the number of Hermite modes $N$ for the
semiclassical limit equation  (Theorem \ref{th:errorSC}).

\subsection{Proof of Theorem \ref{th:01}}
\label{sec:3.2}
Consider the solution $R^\hbar$ to  the von Neumann equation \eqref{vNvW} and
the solution $\scR$ to \eqref{eq:lim}. Then,  the difference $R^\hbar - \scR$ satisfies
\begin{equation*}
\partial_t(R^\hbar-\scR) \,=\,  i\,\left(\partial_x\partial_y \,-\,\frac{V(x+\frac\hbar 2\, y)\;-\,V(x-\frac\hbar 2\,
  y)}{\hbar}\right)\,(R^\hbar-\scR) \,-\, i\,\cE^\hbar\,\scR\,,
\end{equation*}
where $\cE^\hbar$ is defined in \eqref{def:Eh}.

By  Taylor's formula at order $3$, 
$$
V(x+z)\,=\,V(x)+V'(x)\;z\;+\;\frac12\,V''(x)\,z^2\,+\,\frac12\,\int_0^1(1-\theta)^2
\,V^{(3)}(x+\theta
z)\,z^3\,\dD \theta\,,
$$
so that
\begin{eqnarray*}
\cE^\hbar(x,y) &=& \frac{V(x+\tfrac\hbar 2\, y)-V(x-\tfrac\hbar 2\, y)}{\hbar}\,-\,V'(x)\,y
  \\
  &=&
     \frac{\hbar^2}{16}\,y^3\,\int_0^1(1-\theta)^2\left(V^{(3)}(x+\tfrac\hbar 2\,\theta\,
     y)\,+\,V^{(3)}(x-\tfrac \hbar 2\,\theta\, y)\right)\,\dD\theta\,.
\end{eqnarray*}
Hence,
$$
|\cE^\hbar(x,y)|\,=\,\left|\frac{V(x+\tfrac\hbar 2\, y)-V(x-\tfrac\hbar 2\; y)}{\hbar}\,-\,V'(x)\,y\right|\,\le\,\frac{\hbar^2}{24}\,|y|^3\,\|V^{(3)}\|_{L^\infty}\,.
$$
Using this upper bound and applying Proposition
\ref{prop:reg_Rhat} on $\scR$,  it yields that
\begin{eqnarray*}
\|R^\hbar(t)-\scR(t)\|_{L^2} & \le& \|R^\hbar_{\rm in}-\scR_{\rm in}\|_{L^2}\,+\,
\frac{\hbar^2}{24}\,\|V^{(3)}\|_{L^\infty}\,\int_0^t\|y^3\;\scR(s)\|_{L^2}\,\dD s
\\
& \le & \|R^\hbar_{\rm in}-\scR_{\rm in}\|_{L^2} \,+\, \frac{\hbar^2}{24}\,\|V^{(3)}\|_{L^\infty}\,\frac{e^{C_3[V]t}\;-\,1}{C_3[V]}
\;\cN_3(0)\,.
\end{eqnarray*}

\subsection{Proof of Theorem \ref{th:errorSC}}
\label{sec:3.3}

We now consider $\scR$ (resp. $\scR^{N}$), the solution  to the semiclassical equation 
\eqref{eq:lim} (resp. the Galerkin semiclassical equation
\eqref{Hermite:limN}). We will first establish some  properties on $\scR^N$ that
are also satisfied at the continuous level and next show that the $L^2$ estimate given in \eqref{esti:L2} is preserved at
the discrete level when the number of Hermite modes $N$ is finite.

\begin{proposition}
  \label{prop:L2-2}
 For a given $m\geq 0$, assume that $V$ satisfies \eqref{hyp:V1}-\eqref{hyp:V2}
while  $\scR_{\rm in}$ satisfies \eqref{hyp:Rhat}, and consider the solution $\scR^{N}$ to \eqref{Hermite:limN}. Then
\begin{itemize}
\item[$(i)$] for all $k,\, l\in\{0,\ldots, N\}$ and $t\geq 0$, we have 
  \begin{equation}
    \label{p:12}
  \scR^{N}_k(t|\,x) \,=\, (-1)^k\,\overline{\scR^{N}_k}(t|\,x)\,,
\end{equation}
and, assuming that $N$ is even,  the  trace property  in \eqref{TRvW} becomes 
$$
\int_\R  \scR^{N}(t|\,x,0)\,\dD x \,=\, \sum_{k=0}^{N}   \frac{H_k(0)}{\sqrt{2^k\,k!}}
                                                    \,  \int_\R 
                                   \scR^{N}_k(t|\,x) \,\dD x
                                                    \,=\, \int_\R  \cP_{N}\scR_{\rm in} (x,0)\dD x\,;
$$
\item[$(ii)$] for all $t\geq 0$, we have
  \begin{equation}
    \label{p:22}
\| \scR^{N}(t) \|_{L^2_{x,y}} \,=\,  \| \cP_{N}\scR_{\rm in} \|^2_{L^2_{x,y}}\,.
   \end{equation}
\end{itemize}
\end{proposition}
\begin{proof}
 On the one hand, using  \eqref{am:0}, we have
  $$
\Phi_k(y)\,=\,(-1)^k\, \Phi_k(-y),\qquad y\in\R,\quad k\ge 0\,,
  $$
 hence,
  \begin{eqnarray*}
\scR_{{\rm in},k}(x)  \,=\,\int_{\R} \scR_{\rm in}(x,y)\,\Phi_k(y)\,\dD y
                 &=& \int_{\R} \overline{\scR_{\rm in}(x,-y)}\,\Phi_k(y)\,\dD y \\
                 &=& \int_{\R} \overline{\scR_{\rm in}(x,z)}\,\Phi_k(-z)\,\dD z 
                 \,=\, (-1)^k\,\overline{\scR_{{\rm in},k}(x)}\,.
  \end{eqnarray*}
  On the other hand, we easily show that if
  for all $k\in\{0,\ldots,\,N\}$,  $\scR^{N}_k$ is solution to
\eqref{Hermite:limN}, then  $(-1)^k \overline{\scR^{N}_k}$ also
solves \eqref{Hermite:limN}.  By  uniqueness of the solution of \eqref{Hermite:limN},
we get the expected result
$$
\scR^{N}_k(t) \,=\,(-1)^k\,\overline{\scR^{N}_k}(t), \quad
\forall \,t\geq 0\,.
$$
Now let us turn to the second equality, for which we suppose that $N$ is even and compute the time derivative
of the trace of $\scR^{N}$ to find that  
$$
\frac{\dD}{\dD t} \int_\R  \scR^{N}(t|\,x,0)\dD x  \,=\,
                                                         \sum_{k=0}^{N} \frac{H_k(0)}{\sqrt{2^k\,k!}} \,\int_\R 
                                   \partial_t \scR^{N}_k(t|\,x)
                                                         \,\dD x =
  \cJ_1(t) + \cJ_2(t)\,,
$$
with 
$$
\left\{
\begin{array}{l}                                                 
 \ds\cJ_1(t)  \,=\, -i\,\sum_{k=0}^{N} \frac{H_k(0)}{\sqrt{2^k\,k!}}\,\int_\R \left( \sqrt{\frac{k}{2}}\partial_x\scR_{k-1}^{N} -\sqrt{\frac{k+1}{2}}\partial_x\scR_{k+1}^{N}
       \right)\,\dD x\,,
  \\[1.2em]
 \ds\cJ_2(t)  \,=\,  - i\,\sum_{k=0}^{N} \frac{H_k(0)}{\sqrt{2^k\,k!}}\,\int_\R V'(x)\,\left( \sqrt{\frac{k}{2}}\scR_{k-1}^{N} +\sqrt{\frac{k+1}{2}}\scR_{k+1}^{N}
       \right)
                                    \,\dD x
                                                    \,.
\end{array}\right.
$$
We first observe that the  term $\cJ_1(t)$ vanishes since
$\scR^N_k(t|\cdot)$ is sought in $L^2(\R_x)$, hence  $\scR^N_k(t|\,x)\to 0$ along sequences $x_j\to\pm\infty$. Furthermore, using that
$\scR^N_{-1} = \scR^N_{N+1}=0$,   $H_1(0)=0$ while
$H_{N-1}(0)=H_{N+1}(0)=0$ since $N$ is even shows that $\cJ_2(t)$ may be written as,
\begin{eqnarray*}
\cJ_2(t) &=&
                                                    -i\,\sum_{k=1}^{N-1} \int_\R V'(x)\, \scR_{k}^{N}\,\left( \sqrt{\frac{k+1}{2}}\frac{H_{k+1}(0)}{\sqrt{2^{k+1}(k+1)!}} +\sqrt{\frac{k}{2}}  \,\frac{H_{k-1}(0)}{\sqrt{2^{k-1}\,(k-1)!}}
       \right)
                                                     \,\dD x
  \\
                                                 &=&
                                                    -i\,\sum_{k=1}^{N-1}
      \int_\R V'(x)\, \frac{\scR_{k}^{N}}{\sqrt{2^{k+2}k!}}\,\left(
                                                     H_{k+1}(0) +
                                                     \,2\;k\, H_{k-1}(0) \right)
                                                    \,\dD x\,=\, 0,                     
\end{eqnarray*}
since the recursive relation on Hermite polynomials \eqref{def:Hk}  gives $H_{k+1}(0)
+ 2k H_{k-1}(0)=0$. Hence, the result follows when $N$ is even.

We next prove $(ii)$ for any $N\in\N$ using that
\[
\| \scR^{N}(t)\|_{L^2_{x,y}}^2
\,=\,
 \sum_{
	k=0}^{N}\|  \scR_{k}^{N}(t)\|^2_{L^2_x},\qquad t \geq 0\,. 
  \]
Thus, we have 
\begin{eqnarray*}
  \frac{1}{2}\,\frac{\dD}{\dD t} \sum_{k=0}^{N} \| 
                                   \scR^{N}_k(t)\|_{L^2}^2 &=&  {\rm Re}\left(i\,\sum_{k=0}^{N} \left\langle \sqrt{\frac{k}{2}}\cD\scR_{k-1}^{N} +\sqrt{\frac{k+1}{2}}\cD^\star\scR_{k+1}^{N},
       \,\scR^{N}_k\right\rangle\right)\,,
\end{eqnarray*}
where $\langle.,\,\rangle$ denotes the standard inner product
\eqref{inner:p}  on
$L^2(\R;\C)$. Using the duality
property \eqref{prop1:A} on the operators $\cD$ and $\cD^\star$, we prove that the  term on the right-hand side vanishes.  Indeed, we have
\begin{eqnarray*}
 \sum_{k=0}^{N}\left\langle\sqrt{\frac{k}{2}}\cD
                                                          \scR_{k-1}^{N} +\sqrt{\frac{k+1}{2}}\cD^\star \scR_{k+1}^{N}
       ,  \,\scR^{N}_k\right\rangle
                                                              =
                                                                  \sum_{k=0}^{N} \left(\sqrt{\frac{k}{2}}
                                                          \left\langle
  \scR_{k-1}^{N},\,
                                                                  \cD^\star
                                                                  \scR^{N}_k\right\rangle
                                                                  +
                                                                 \sqrt{\frac{k+1}{2}}
                                                                  \left\langle\cD^\star
  \scR_{k+1}^{N},\, \scR^{N}_k\right\rangle\right).
  \end{eqnarray*}
  Reordering and recalling that $\scR^N_{-1}=\scR^N_{N+1}=0$, we find that
  \begin{eqnarray*}
     \sum_{k=0}^{N} \left(\sqrt{\frac{k}{2}}
                                                          \left\langle
    \scR_{k-1}^{N},\,
                                                                  \cD^\star
                                                                  \scR^{N}_k\right\rangle
                                                                  +
                                                                 \sqrt{\frac{k+1}{2}}
                                                                  \left\langle\cD^\star\scR_{k+1}^{N},\, \scR^{N}_k\right\rangle\right)
                                                              &=&
                                                                  \Sigma
                                                                  +
                                                                  \overline{\Sigma}
                                                                  \,\in\R\,,
 \end{eqnarray*}
 with
 $$
 \Sigma \,=\, \sum_{k=1}^{N} \sqrt{\frac{k}{2}}\,\langle\scR_{k-1}^{N},\,\cD^\star\scR^{N}_k\rangle\,.
 $$                                                         Hence
\[
\frac{\dD}{\dD t}\sum_{k=0}^n\|\scR^N_k(t)\|^2
=2\mathrm{Re}\left(i\left(\Sigma
+\overline{\Sigma}\right)\right)=0,
\]
so that, the
                                                          $L^2$ norm
                                                          for
                                                          \eqref{Hermite:limN}
                                                          is  conserved.  
\end{proof}

We are now ready to prove Theorem \ref{th:errorSC}.  We consider  $\scR$ the solution to the semiclassical equation \eqref{eq:lim} with $d=1$ and
$\scR^N$ the solution to the semiclassical Galerkin system
\eqref{Hermite:limN} and recall that $\cP_N\scR$ is
the projection of the semiclassical solution $\scR$ onto
the first $N$ Hermite modes. From the triangle inequality, the error estimate is obtained by bounding  one term corresponding
to the projection error and  a second one coming from
the error between  $\scR^{N}$ and $\cP_N\scR$, that is,
\begin{equation}
  \label{goal:1}
\| \scR^N(t) - \scR(t) \|_{L^2} \,\leq \, \|  \scR^N(t) -  \cP_N\scR(t) \|_{L^2} \,+\, \| \cP_N\scR(t) - \scR(t) \|_{L^2}\,.  
\end{equation}
We first evaluate the propagation error between the solution $\scR^{N}$
to \eqref{Hermite:limN} and the projection $\cP_{N} \scR$, where $\scR$ is the solution  to semiclassical limit equation \eqref{eq:lim}. For all $k\in\{0,\ldots,N\}$,
$$
\begin{aligned}
  \partial_t\left(\scR^{N}_k-\scR\right)_k \,=\,  \,-\, i\,\left(\sqrt{\frac k 2}\,
      \cD\left(\scR^{N}-\scR\right)_{k-1}\,+\, \sqrt{\frac{k +1}{2}}\,
      \cD^\star \left(\scR^{N}-\scR \right)_{k+1}\right)\,,
\end{aligned}
$$
and  we deduce from  the Duhamel formula that
$$
\scR^{N}(t)-\cP_N\scR(t) \,=\, \cS_0(t) \left(\scR^{N}-\cP_N\scR\right)(0)
\,+\, i\, \sqrt{\frac{N +1}{2}}\,\int_0^t  \cS_0(t-s) 
      \cD^\star\scR_{N+1}(s) \,\dD s
$$
where $\cS_0(t)$ is the one-parameter group \eqref{def:S0}. Hence, using the conservation of the $L^2$ norm  by $\cS_0(t)$, we obtain that
\begin{eqnarray}
  \label{star}
  \|\scR^{N}(t)-\cP_{N} \scR(t)\|_{L^2} &=& \left(\sum_{k=0}^{N}\int_\R|\scR_k^{N}(t|\,x)-\scR_k(t|\,x)|^2\dD x\right)^{1/2}
  \\
  \nonumber
                                     &\leq& 
  \sqrt{\frac{N+1}2}\int_0^t\left(\int_\R\left|\cD^\star\scR_{N+1}(\tau|\,x) \right|^2\dD
                                            x\right) ^{1/2} \dD\tau\,.
\end{eqnarray}
Let us estimate the right-hand side using  the  Parseval identity :
since $\cD^* $ acts on the variable $x$, while $\cP_N$ acts on the
variable $y$,  these two operators commute, and applying   Lemma
\ref{lem:4} shows that
\begin{eqnarray*}
\int_\R\left| \cD^\star\scR_{N+1}(\tau|\,x)\right|^2\dD x &\leq &
                                                    \iint_{\R^2}\left|
                                                                    \cD^\star\left(\scR-\cP_{N}\scR\right)(\tau|\,x,y)\right|^2\dD
                                                    x\,\dD y
\\
&\le&
      \frac1{(2{N}+3)^{2p}}\,\iint_{\R^2}\left| \left(y^2-\partial^2_y\right)^p\cD^\star\scR(\tau|\,x,y)\right|^2\dD x\,\dD y\,.
\end{eqnarray*}
Then, using the mean-value inequality \eqref{my:num3}
\begin{eqnarray*}
\left|\left(y^2-\partial^2_y\right)^p\cD^\star\scR (\tau|\,x,y)\right|^2 &\leq& 2\,\left| \left(y^2-\partial^2_y\right)^p\partial_x\scR (\tau|\,x,y)\right|^2\,+\, 2\,|V'(x)|^2\, \left| \left(y^2-\partial^2_y\right)^p\scR (\tau|\,x,y)\right|^2
\\
&\le&  2\,\left| \left(y^2-\partial^2_y\right)^p\partial_x\scR (\tau|\,x,y)\right|^2
\\
                                            &+& 4 \left( |V'(0)|^2 +
                                                x^2 \,
                                                \|V''\|_{L^\infty}^2\right)
                                                \,  \left| \left(y^2-\partial^2_y\right)^p\scR (\tau|\,x,y)\right|^2.
\end{eqnarray*}
Hence, there exists a constant $C_p>0$, depending only on $p\geq 1$
such that, for all $\tau\geq 0 $,
$$
\iint_{\R^2} \left|\left(y^2-\partial^2_y\right)^p\cD^\star\scR(\tau|\,x,y)\right|^2\dD x\,\dD y
\,\le\, C_p\max\left(1,|V'(0)|,\|V''\|_{L^\infty}\right)\,\cN_{2p+1}(\tau)^2\,,
$$
where $\cN_{2p+1}(\tau)$ is given in \eqref{def:Nm}.  Therefore, we
conclude by applying Proposition \ref{prop:reg_Rhat} that
\begin{eqnarray}
  \nonumber
&&\int_0^t\left(\int_\R|\cD^\star\scR_{N+1}(\tau|\,x)|^2\dD x\right)^{1/2}\dD\tau
  \\[0.9em]
  \label{bistar}
  &\le & \frac1{(2{N}+3)^p}\,\int_0^t\left(\iint_{\R^2}\left|\left(y^2-\partial^2_y\right)^p\cD^\star\scR(\tau|\,x,y)\right|^2\dD x\,\dD y\right)^{1/2}\dD\tau
\\[0.9em]
&\le&\frac{\sqrt{C_p  \,\max\left(1,|V'(0)|,\|V''\|_{L^\infty}\right)}}{(2{N}+3)^p}\,\frac{e^{C_{2p+1}[V]t}-1}{C_{2p+1}[V]}\,\cN_{2p+1}(0)\,,
\end{eqnarray}
where $C_{2p+1}[V]$ is given in \eqref{def:Cm}. Finally, gathering
together \eqref{star}-\eqref{bistar}, it yields the following inequality  
\begin{equation}
  \label{res:01}
\|\scR^{N}(t)-\cP_{N} \scR(t)\|_{L^2}\,\le\,
\frac{\sqrt{C_p\,\max(1,|V'(0)|,\|V''\|_{L^\infty})}}{(2{N}+3)^{p-1/2}}\,\frac{e^{C_{2p+1}[V]t}-1}{C_{2p+1}[V]}
\, \cN_{2p+1}(0)\,.
\end{equation}

Now, it remains to evaluate  the last term on the right-hand side in
\eqref{goal:1}, which corresponds to the projection error.  Thanks to Lemma \ref{lem:4}, we have
$$
\|\scR(t)-\cP_{N}\scR(t)\|_{L^2(\R^2)}\le\frac1{(2{N}+3)^p}\|\left(y^2-\partial_y^2\right)^p\scR(t)\|_{L^2(\R^2)},
$$
hence using the propagation of regularity in Proposition \ref{prop:reg_Rhat}, we
finally get
\begin{equation}
  \label{res:02}
\|\scR(t)-\cP_{N}\scR(t)\|_{L^2(\R^2)}\,\le\,\frac{K_p}{(2{N}+3)^p} \, e^{C_{2p}[V]\,t}\,\cN_{2p}(0),
\end{equation}
where $K_p>0$ is a constant depending only on $p$. Gathering together the former results \eqref{res:01} and \eqref{res:02}, we get
the desired inequality
\begin{eqnarray*}
\| \scR^N(t) - \scR(t) \|_{L^2} & \leq &   \frac{K_p}{(2{N}+3)^p} \,
                                             e^{C_{2p}[V]\,t}\,\cN_{2p}(0)
  \\
  &+& \frac{\sqrt{C_p\,\max(1,|V'(0)|,\|V''\|_{L^\infty})}}{(2{N}+3)^{p-1/2}}\,\frac{e^{C_{2p+1}[V]t}-1}{C_{2p+1}[V]}
\, \cN_{2p+1}(0)\,.
\end{eqnarray*}

\section{Analysis of the von Neumann equation}
\label{sec:4}
\setcounter{equation}{0}
\setcounter{figure}{0}
\setcounter{table}{0}

 We now treat the von Neumann equation written in Weyl's variables
\eqref{vNvW} and focus on the contribution of the non-local term
$\cE^\hbar$ in \eqref{def:Eh}. We apply the same strategy as before, but the proof of
the propagation of regularity now becomes more intricate.

\subsection{Propagation of regularity on the von Neumann equation}
\label{sec:4.1}

It is worth mentioning that the latter result only requires that the
solution to the von Neumann equation \eqref{vNvW} is $R^\hbar(t)\in L^2$
since we use the regularity of the solution $\scR$ of \eqref{eq:lim} proved in Proposition \ref{prop:reg_Rhat}.  However, when
the initial datum $R^\hbar_{\rm in}$ is smooth, we  prove that for any $m\geq
0$, the quantity $\cN_m(t)$ is also  propagated for the
solution $R^\hbar$ to the von Neumann equation \eqref{vNvW}.

\begin{proposition}
  \label{prop:reg_R}
Let $m\geq 0$ and  $V$ be the potential  such that
\eqref{hyp:V1}-\eqref{hyp:V2} hold, and
suppose that  for all integer  $a$, $b$,
$\alpha$ and $\beta$ such that  $a+b+\alpha+\beta\le m$, the
assumption \eqref{hyp:R}  on the initial data $R^\hbar_{\rm in}$ is satisfied. Then the solution $R^\hbar$ to \eqref{vNvW}  is such that  for all  $t\ge 0$,
\begin{equation*}
\cN_m(t)
\le \cN_m(0) \,e^{C_m[V]\,t}\,,
\end{equation*}
where $\cN_m(t)$ is defined in \eqref{def:Nm} and applied to
$R^\hbar$ while $C_m[V]>0$ is a constant depending only on $m$ and on the potential
$V$.
\end{proposition}

\begin{proof}
  This proof follows the same strategy as the one presented in
  Proposition \ref{prop:reg_Rhat} but  is rather more technical. See Appendix \ref{appA} for details. 
  \end{proof}

\subsection{Proof of Theorem \ref{th:errorvN}}
\label{sec:4.2}

Let us first show that the $L^2$ estimate given in \eqref{esti:L2} is
preserved when the number of modes $N$ is finite in the approximation
of the von Neumann equation using Hermite polynomials.

\begin{proposition}
  \label{prop:L2-1}
For a given $m\geq 0$, assume that $V$ satisfies \eqref{hyp:V1}-\eqref{hyp:V2}
while  $R^\hbar_{\rm in}$ satisfies \eqref{hyp:R} and consider $R^{\hbar,N}$ the solution to \eqref{Hermite:DN}. Then
\begin{itemize}
\item[$(i)$] for all $k,\, l\in\{0,\ldots, N\}$ and $t\geq 0$, we have 
  \begin{equation}
    \label{p:1}
  R^{\hbar,N}_k(t) \,=\, (-1)^k\,\overline{R^{\hbar,N}_k}(t)\,;
  \end{equation}
\item[$(ii)$] for all $t\geq 0$, we have
  \begin{equation}
    \label{p:2}
\| R^{\hbar,N}(t) \|_{L^2} \,=\,  \| \cP_{N}R^\hbar_{\rm in} \|^2_{L^2}\,.
   \end{equation}
\end{itemize}
\end{proposition}
\begin{proof}
 To prove \eqref{p:1}, use \eqref{toto:1} and proceed as in the proof
 of  Proposition \ref{prop:L2-2}  to show that, if for all
 $k\in\{0,\ldots,N\}$, $R^{\hbar,N}_k$ is solution to  \eqref{Hermite:DN}, then
 $(-1)^k \, \overline{R^{\hbar,N}_k}$  also solves \eqref{Hermite:DN},
 which gives the expected result.
 \\
 
We next prove the propagation of the $L^2$ norm  for $R^{\hbar,N}$ again  as
in Proposition \ref{prop:L2-2}. Thus, we have 
$$
  \frac{1}{2}\,\frac{\dD}{\dD t}  \| 
  R^{\hbar,N}(t)\|_{L^2}^2 \,=\, \cI_1(t) + \cI_2(t),
$$
with
$$
\left\{
\begin{array}{l}
  \ds\cI_1(t) \,:=\, {\rm Re}\left( i\,\sum_{k=0}^{N}
                                                          
                                                          \left\langle
                                                          \sqrt{\frac{k}{2}}\cD
                                                          R^{\hbar,N}_{k-1} +\sqrt{\frac{k+1}{2}}\cD^\star R^{\hbar,N}_{k+1}
      , \, R^{\hbar,N}_k
                                                               \right\rangle \right)\,,
  \\[1.1em]
 \ds\cI_2(t) \,:=\, {\rm Re}\left( i\,\sum_{k,l=0}^{N} \langle \cE^\hbar_{k,l} \,
R^{\hbar,N}_l,\,  R^{\hbar,N}_k\rangle\,\right)\,,
\end{array}
\right.
$$
where $\langle.,\,\rangle$ denotes the standard inner product
\eqref{inner:p}  on
$L^2(\R;\C)$. On the one hand,  observing that $\cI_1(t)$ is the same
term as  the one of the right-hand side in the proof of
Proposition \ref{prop:L2-2}, we get that $\cI_1(t) =0$. On the other
hand, using that $\cE^\hbar_{k,l}\in \R$ is symmetric  and satisfies
\eqref{toto:1} while $R^{\hbar,N}$ verifies  \eqref{p:1}, we get
that for all $k$, $l\in\{0,\ldots, {N}\}$
\begin{eqnarray*}
\langle \cE^\hbar_{k,l} \,
R^{\hbar,N}_l ,\, R^{\hbar,N}_k\rangle  &=& \int_{\R} \cE^\hbar_{k,l}\,R^{\hbar,N}_l\,\, \overline{R^{\hbar,N}_k}\,\dD x
\\
&=&\int_{\R}\left(  - (-1)^{l+k}\cE^\hbar_{k,l}\right)\,\left( (-1)^l\,\overline{R^{\hbar,N}_l}\right)\,\,\left((-1)^k R^{\hbar,N}_k\right)\,\dD x
  \\
                 &=&   -\,\int_{\R}\cE^\hbar_{l,k}\, R^{\hbar,N}_k\,\, \overline{R^{\hbar,N}_l} \,\dD x  
                     \,=\, -\langle \cE^\hbar_{l,k} \,R^{\hbar,N}_k,\,  R^{\hbar,N}_l \rangle\,.
\end{eqnarray*}
Therefore, $\langle \cE^\hbar_{k,l}R^{\hbar,N}_l, \, R^{\hbar,N}_k\rangle$ is
antisymmetric in $(k,l)$ so that
\begin{eqnarray*}
\sum_{k,l=0}^{N} \langle \cE^\hbar_{k,l} \,R^{\hbar,N}_l , \, R^{\hbar,N}_k\rangle
       \,=\, 0\,,
\end{eqnarray*}
hence $\cI_2(t)$ also vanishes, leading to the conservation of the $L^2$ norm for $R^{\hbar,N}$. 
\end{proof}

%
%

We are now ready to complete the proof of Theorem \ref{th:errorvN} and
follow the lines of the arguments provided in Section \ref{sec:3.2}
for Theorem \ref{th:errorSC}. Here we will only
focus on the additional term $\cE^\hbar$ when evaluating the propagation error between the solution $R^{\hbar,N}$
to \eqref{Hermite:DN} and the projection $\cP_{N} R^\hbar$, where $R^\hbar$ is the
solution  to  the von Neumann equation\eqref{vNvW}.   We need to estimate the nonlocal term
involving $\cE_{k,l}^\hbar$, namely,
$$
\cI_N(t) \,=\, \sum_{k=0}^N \sum_{l> N} \int_\R   \cE_{k,l}^\hbar(x)\,
             \,R^\hbar_l(t|\,x) \,\,\overline{\left(R^\hbar_k - R^{\hbar,N}_k\right)}(t|\,x) \,\dD x\,.
           $$
           Let us re-write $\cI_N(t)$ as
\begin{eqnarray*}
 \cI_N(t)  &=& \iint \cE^\hbar(x,y)\, \overline{\left(\cP_NR^\hbar-R^{\hbar,N}\right)}(t|\,x,y)  \, \left( I-
  \cP_N\right)\,R^\hbar(t|\,x,y) \,\dD
               y\,\dD x
  \\
   &\leq& \|\cE^\hbar\, \left(\cP_NR^\hbar(t) - R^{\hbar,N}(t)\right) \|_{L^2}\, \| R^\hbar(t) -\cP_NR^\hbar(t)\|_{L^2_{x,y}} \,\,. 
\end{eqnarray*}
           Using that
           $$
|\cE^\hbar(x,y)|\,=\,\left|\frac{V(x+\tfrac\hbar 2\, y)-V(x-\tfrac\hbar 2\; y)}{\hbar}\,-\,V'(x)\,y\right|\,\le\,\frac{\hbar^2}{24}\,|y|^3\,\|V^{(3)}\|_{L^\infty}\,,
$$
we have 
\begin{eqnarray*}
  && \|\cE^\hbar\, \left(R^{\hbar,N}(t)-\cP_NR^\hbar(t)\right) \|_{L^2_{x,y}}^2
\,=\,  \iint  |\cE^\hbar(x,y)|^2 \,\left|\left(R^{\hbar,N}-\cP_NR^\hbar\right)(t|\,x,y) \right|^2 \dD y \,\dD x
\\
  &&\leq\,
                                                             \left(\frac{\hbar^2}{24}\,\|V^{(3)}\|_{L^\infty}\right)^2 \,\sum_{j,k=0}^N\int
              \left(R^{\hbar,N}_k-R^\hbar_k\right)\,\overline{\left(R^{\hbar,N}_j-R^\hbar_j\right)}(t|\,x)\,\dD x \,  \int |y|^6 \,\Phi_k(y) \,\Phi_j(y)\,
                                                             \dD y \,.
\end{eqnarray*}
Now using repeatedly that
\begin{equation*}
y\,\Phi_k(y)\,=\, \sqrt{\frac{k+1}{2}}\,\Phi_{k+1}(y) \,+\, \sqrt{\frac{k}{2}}\,\Phi_{k-1}(y) \,,\quad\, k\,\geq\,0\,,
\end{equation*}
we have
$$
y^3 \Phi_k(y) \,=\, \alpha_k \,  \Phi_{k+3}(y) +\beta_k \,
\Phi_{k+1}(y) + \gamma_k \,  \Phi_{k-1}(y) +\delta_k \,  \Phi_{k-3}(y),
$$
with
$$
\left\{
  \begin{array}{l}
    \ds\alpha_k \,=\, \sqrt{\frac{(k+1)(k+2)(k+3)}{8}}\,,\\[0.9em]
    \ds\beta_k \,=\, \frac{3(k+1)}{2}\sqrt{\frac{k+1}{2}}\,,\\[0.9em] 
    \ds\gamma_k \,=\, \frac{3k}{2}\sqrt{\frac{k}{2}}\,,\\[0.9em]
    \ds\tau_k \,=\, \sqrt{\frac{k(k-1)(k-2)}{8}}\,,
  \end{array}\right.
$$
so that
$$
\alpha_k,\, \beta_k,\, \gamma_k,\, \tau_k\,\leq\, \frac{3}{2\sqrt2}\,(k+3)^{3/2}\,.
$$
Hence, we obtain
\begin{eqnarray*}
\int_\R |y|^6 \Phi_j(y)\,\Phi_k(y)\,\dD y &=&  \int_\R |y|^3
                                              \Phi_j(y)\,
                                              |y|^3\Phi_k(y)\,\dD y \\[0.9em]
  &=& \left( \alpha_j^2
                                              +\beta_j^2 +\gamma_j^2+ 
                                              \tau_j^2  \right)
                                              \,\delta_{j,k}
  \\[0.9em]
  &+& \ds\left( \alpha_k\beta_j
                                              +\beta_k\gamma_j + 
                                              \gamma_k\tau_j  \right)
                                              \,\delta_{j,k+2} \,+\,
  \left(\alpha_k\gamma_j + \beta_k\tau_j\right)\, \delta_{j,k+4} \,+\,
  \alpha_k\tau_j\, \delta_{j,k+6}
  \\[0.9em]
                                          &+&
                                             \ds \left( \alpha_j\beta_k
                                              +\beta_j\gamma_k + 
                                              \gamma_j\tau_k  \right)
                                              \,\delta_{j+2,k}  \,+\,
  \left(\alpha_j\gamma_k + \beta_j\tau_k\right)\, \delta_{j+4,k} \,+\,
  \alpha_j\tau_k\, \delta_{j+6,k}\,,
\end{eqnarray*}
and for all $0\leq j,\,k\leq N$, denoting $\max(j,k)$ by $j\vee k$,
$$
a_j\, b_k \,\leq\, \frac{9}{8}\,(j\vee k +9)^3, \qquad\forall\,a,\,b\,\in\{\alpha,\,\beta,\,\gamma\,,\tau\}. 
$$
Finally, we get
\begin{eqnarray*}
  \|\cE^\hbar\, \left(R^{\hbar,N}(t)-\cP_NR^\hbar(t)\right) \|_{L^2}^2&\leq&
                                                             \frac{1}{8}\,\left(\frac{\hbar^2}{8}\,\|V^{(3)}\|_{L^\infty}\right)^2\times  \left(                                                     \sum_{j=0}^N
                                                                 (j+9)^3\int
                                                               |R^{\hbar,N}_j
                                                                             -
                                                                             R^\hbar_j|^2
                                                               (t|\,x)\,\dD
                                                               x \right.
  \\
  && +\,  2\,\sum_{j=0}^{N-2}
                                                                 (j+11)^3\int
                                                               |R^{\hbar,N}_j
     -R^\hbar_j|\,   |R^{\hbar,N}_{j+2}-R^\hbar_{j+2}|
                                                               (t|\,x)\,\dD x
  \\ && +\,  2\,\sum_{j=0}^{N-4}
                                                                 (j+13)^3\int
                                                               |R^{\hbar,N}_j-R^\hbar_j|\,   |R^{\hbar,N}_{j+4}-R^\hbar_{j+4}|
                                                               (t|\,x)\,\dD  x  
  \\
  && \left. +\,  2\,\sum_{j=0}^{N-6}
                                                                 (j+15)^3\int
                                                               |R^{\hbar,N}_j
     -R^\hbar_j|\,   |R^{\hbar,N}_{j+6}-R^\hbar_{j+6}|
                                                               (t|\,x)\,\dD
     x\right)
  \\
  &\leq &
          \frac{7\,(N+9)^3}{2}\,\left(\frac{\hbar^2}{8}\,\|V^{(3)}\|_{L^\infty}\right)^2\,
          \|R^{\hbar,N}(t)-\cP_N R^\hbar(t)\|_{L^2}^2.
\end{eqnarray*}
Now we are ready to bound the term $\cI_N(t)$ by applying the latter estimate and Lemma \ref{lem:4}, which gives for $N\geq 6$,
\begin{eqnarray*}
\cI_N(t) &\leq&
                \frac{\hbar^2}{4}\,\|V^{(3)}\|_{L^\infty}\,(N+9)^{3/2} \,
                \| R^{\hbar,N}(t) - \cP_NR^\hbar(t)\|_{L^2}\, \|R^\hbar(t)-\cP_NR^\hbar(t)\|_{L^2}
  \\
  &\leq&   \frac{\hbar^2}{4}\, \|V^{(3)}\|_{L^\infty}\,
                \frac{(N+9)^{3/2}}{(2{N}+3)^{p+1}}
         \,\|\left(y^2-\partial_y^2\right)^{p+1}R^\hbar(t)\|_{L^2}\, \|
         R^{\hbar,N}(t) - \cP_NR^\hbar(t)\|_{L^2}
  \\
  &\leq&   \frac{\hbar^2}{4}\, \|V^{(3)}\|_{L^\infty}\, \frac{1}{(2N+3)^{p-1/2}}
                 \,\|\left(y^2-\partial_y^2\right)^{p+1}
         R^\hbar(t)\|_{L^2}\, \| R^{\hbar,N}(t) - \cP_NR^\hbar(t)\|_{L^2}
         \\
  &\leq&   \frac{\hbar^2}{4}\, \|V^{(3)}\|_{L^\infty}\, \frac{ K_p \,\cN_{2p+2}(t) }{(2N+3)^{p-1/2}}
                 \,\| R^{\hbar,N}(t) - \cP_NR^\hbar(t)\|_{L^2}\,,
\end{eqnarray*}
where $K_p$ is a constant depending on $p$ only and $\cN_{2p+2}(t)$ is evaluated in Proposition \ref{prop:reg_R}.

Finally, gathering together this latter result and those obtained in
the proof of Theorem \ref{th:errorSC} on the contribution of $\cD^\star
R^\hbar_{N+1}$, we are lead to the  following inequality
\begin{eqnarray*}
\frac{\dD}{\dD t}\| R^{\hbar,N}(t) - \cP_N R^\hbar(t) \|_{L^2} & \leq &  \frac{\sqrt{C_p\,\max(1,|V'(0)|,\|V''\|_{L^\infty})}}{(2{N}+3)^{p-1/2}}\, \cN_{2p+1}(t) 
    \\
  &+& \frac{\hbar^2}{4}\,\|V^{(3)}\|_{L^\infty}\, \frac{2\,K_p\,\cN_{2p+2}(t)}{(2N+3)^{p-1/2}}\,.
\end{eqnarray*}
Thus, we have
\begin{eqnarray*}
\| R^{\hbar,N}(t) - \cP_N R^\hbar(t) \|_{L^2}  &\leq& \frac{\sqrt{C_p\,\max(1,|V'(0)|,\|V''\|_{L^\infty})}}{(2{N}+3)^{p-1/2}}\,\frac{e^{C_{2p+1}[V]t}-1}{C_{2p+1}[V]}\, \cN_{2p+1}(0)
  \\
  &+& \frac{2\,K_p\,\hbar^2}{4\, (2N+3)^{p-1/2}}\,\|V^{(3)}\|_{L^\infty}\, \,\frac{e^{C_{2p+2}[V]t}-1}{C_{2p+2}[V]}\, \cN_{2p+2}(0)\,.
\end{eqnarray*}
Finally from this latter inequality and the estimate on the norm of the operator $I-\mathcal P_N$ in Lemma \ref{lem:4}, on
the operator $I-\cP_N$, we get the sought inequality.


\section{Numerical simulations}
\label{sec:5}
\setcounter{equation}{0}
\setcounter{figure}{0}
\setcounter{table}{0}

{We now perform a few numerical simulations to illustrate the spectral accuracy of  the
proposed  Hermite  method for the approximation
of the solution to the von Neumann equation \eqref{vNvW} in the case of a quartic
potential $V$. We refer to \cite{FG1:24} for more extensive numerical
simulations for various potential. 
\\
Thus, we  consider the  isotropic quartic potential
$$
V(x) \,:=\, \frac{x^2}{2} \,+\,  \chi\,\frac{x^4}{4},
$$
with $\chi=1/2$, in order to benchmark the convergence rate of the proposed
method.  By injecting this potential in the  von Neumann equation \eqref{vNvW}, one
can easily verify that the von Neumann equation takes the simple form:
$$
\partial_t R^\hbar \,=\,i \left(\partial_{xy} R^\hbar \,-\,  V^\prime(x) \,y R^\hbar\,-\, \chi\,\frac{\hbar^2}{4} \,x\,y^3 R^\hbar\right)\,
$$
since derivatives of order larger than  five of the potential  vanish
identically.  We then choose the initial datum
\cite{manfredi1996theory},  which is a coherent state (in appropriate coordinates)
$$
R_{\rm in}^\hbar(x,y) \,=\,\frac{1}{\sqrt{2\pi}\sigma_x} \exp\left( -
 \frac{1}{2} \left(\frac{x^2}{\sigma_x^2} + y^2 \right) \right)\,,
$$
with $\sigma_x=0.6$ and $\hbar=0.1$.}

In this situation, there is no explicit solution to the von Neumann
equation \eqref{vNvW}, but we compute a reference solution
$R^\hbar_{{\rm ref}}$ using a fine mesh  $N=500$. We also choose the $x$ domain as $[-4,4]$ and  mainly focus on the
convergence with respect to $N$ taking the time step $\Delta t$ and
the space step $\Delta x$ sufficiently small, $\Delta t=10^{-4}$ and
$\Delta x = 10^{-3}$.  Numerical errors  are
presented in Table \ref{tab:1}, where the numerical error is given as
$$
\cE(N) \,:=\, \max_{t\in [0,2\pi]}\left(\sum_{k=0}^{N}\|
  R^{\hbar,N}_{k}(t) - R^\hbar_{{\rm ref}, k}(t^n) \|^2_{L^2_x}\right)^{1/2}.
$$

  \begin{figure}[!htbp]
  \begin{minipage}[b]{.5\linewidth}
    \centering
    \includegraphics[width=9.cm]{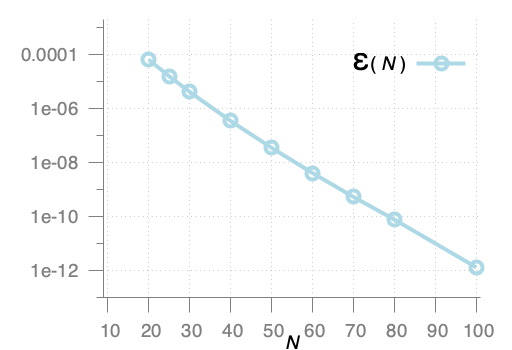}
    \captionof{figure}{$L^2$ error in log scale.}
   \end{minipage}\hfill
   \begin{minipage}[b]{.5\linewidth}
    \centering
    \begin{tabular}{|l|l|l|}
          \hline
	Number of  &  Discrete $L^2$  & Order of \\[0.35em] 
      modes $N $  &  error  $\cE(N)$  & accuracy \\[0.45em]  \hline
	$20$  &  $6.3316\, \times\,10^{-05}$     & X    \\[0.45em]  \hline
	$30$  &  $4.0451\, \times\,10^{-06}$     & 6.8    \\[0.45em]  \hline
        $40$  &  $3.3937\, \times\,10^{-07}$     & 8.6  \\[0.45em]  \hline
        $50$  &  $3.4347\, \times\,10^{-08}$     & 10     \\[0.45em]  \hline
        $60$  &  $3.9989\, \times\,10^{-09}$   & 11     \\[0.45em]  \hline		
        $70$  &  $5.2016\, \times\,10^{-10}$   & 13     \\[0.45em]  \hline		
    \end{tabular}
    \captionof{table}{$L^2$ error with respect to the number of Hermite modes $N$.}
    \label{tab:1}
  \end{minipage}
\end{figure}

\section{Conclusion}
\label{sec:6}
\setcounter{equation}{0}
\setcounter{figure}{0}
\setcounter{table}{0}
In conclusion, we have introduced a new approach to discretize the
von Neumann equation in the semiclassical limit. By using Weyl's
variables and  a truncated Hermite expansion of the density
operator, we are able to address the stiffness due to the semiclassical scaling of the equation. Our method allows for error estimates by leveraging
the propagation of regularity on the exact solution. Therefore, we believe that our asymptotic
preserving numerical approximation, coupled with Hermite polynomials,
represents an interesting approach for accurately solving the von
Neumann equation in the semiclassical regime. 

{In this article, we focus on the one-dimensional case to avoid
increasing technical complexity. However, there is no doubt that this
analysis can be extended to higher dimensions, provided that
appropriate regularity assumptions are made on the external potential
and the initial data.}

\appendix
\section{Proof of Proposition \ref{prop:reg_R}}
\label{appA}
The case $m=0$ corresponds to the conservation of the $L^2$ norm
$$
\| R^\hbar(t) \|_{L^2} \,=\,  \| R^\hbar_{\rm in} \|_{L^2}\,,\qquad t \geq 0. 
$$
Let us begin with the propagation of quantities corresponding to
phase-space derivatives in the Wigner equation, that is to say  
$\partial_xR^\hbar$ and $yR^\hbar$. Later, it will be convenient to use the notation
\[
L\,:=\,\max(1,\|V''\|_{L^\infty})\,.
\]
On the one hand, the equation for the $x$-derivative of $R^\hbar$ is given by
\begin{eqnarray*}
\partial_t\partial_xR^\hbar(t|\,x,y) &=& i\,\partial_x\partial_y\partial_xR^\hbar(t|\,x,y)\,-\,i\,\partial_x\left(\frac{V(x+\frac12\hbar y)-V(x-\frac12\hbar y)}{\hbar}R^\hbar(t|\,x,y)\right)
\\
&=& i\,\partial_x\partial_y\partial_xR^\hbar(t|\,x,y)\,-\,i\,\frac{V(x+\frac12\hbar y)-V(x-\frac12\hbar y)}{\hbar}\partial_xR^\hbar(t|\,x,y)
\\
&-&\,i\,\frac{V'(x+\frac12\hbar y)\,-\,V'(x-\frac12\hbar y)}{\hbar}R^\hbar(t|\,x,y)\,,
\end{eqnarray*}
so that
\[
\partial_xR^\hbar(t|\,x,y)\,=\,\cS_\hbar(t)\partial_xR^\hbar(0|x,y)\,-\,i\,\int_0^t\cS_\hbar(t-s)\frac{V'(x+\frac12\hbar y)-V'(x-\frac12\hbar y)}{\hbar}\,R^\hbar(s|\,x,y)\,\dD s\,,
\]
where $\cS_\hbar(t)$ represents the one-parameter group \eqref{def:S} associated
to \eqref{vNvW}. Since $\cS_\hbar(t)$ is an $L^2_{x,y}$ isometry,
\begin{eqnarray*}
\|\partial_xR^\hbar(t)\|_{L^2}
  &\le&\|\partial_xR^\hbar(0)\|_{L^2}\,+\,\int_0^t \left\|\frac{V'(x+\frac12\hbar y)-V'(x-\frac12\hbar y)}{\hbar\, y}\,y\,R^\hbar(s|\,x,y)\right\|_{L^2}\,\dD s
\\
&\le&\|\partial_xR^\hbar(0)\|_{L^2}+\|V''\|_{L^\infty}\,\int_0^t\|y\,R^\hbar(s)\|_{L^2}\dD s\,.
\end{eqnarray*}
Similarly, 
\begin{eqnarray*}
\partial_t\left(y\,R^\hbar(t|\,x,y)\right) &=& i\,\partial_x\partial_y\left(y\,R^\hbar(t|\,x,y)\right)\,-\,i\,\frac{V(x+\frac12\hbar y)-V(x-\frac12\hbar y)}{\hbar}\left(y\,R^\hbar(t|\,x,y)\right)
\\
&-& i\,\partial_xR^\hbar(t|\,x,y)\,,
\end{eqnarray*}
so that, arguing as above, 
\[
\|y\,R^\hbar(t)\|_{L^2}\le\|y\,
R^\hbar(0)\|_{L^2} \,+\,\int_0^t\|\partial_xR^\hbar(s)\|_{L^2}\dD s\,.
\]
Summarizing, we have proved that
\begin{eqnarray*}
\|\partial_xR^\hbar(t)\|_{L^2}&+&\|y\,
  R^\hbar(t)\|_{L^2} - \|\partial_xR^\hbar(0)\|_{L^2}\,-\,\|y\,R^\hbar(0)\|_{L^2}
\\
&\le &\int_0^t\left(\|V''\|_{L^\infty}\,\|y\,R^\hbar(s)\|_{L^2}\,+\,\|\partial_xR^\hbar(s)\|_{L^2}\right)\dD s
\\
&\le& L\,\int_0^t\left(\|\partial_xR^\hbar(s)\|_{L^2}\,+\,\|y\,R^\hbar(s)\|_{L^2}\right)\dD s\,.
\end{eqnarray*}
Applying Gronwall's lemma, implies that
\begin{equation}
  \label{<1}
\|\partial_xR^\hbar(t)\|_{L^2}\,+\,\|y\,R^\hbar(t)\|_{L^2}
\,\le\,\left(\|\partial_xR^\hbar(0)\|_{L^2}\,+\,\|y\,R^\hbar(0)\|_{L^2}\right)\,e^{L\,t}\,.
\end{equation}
Here we observe that this estimate is uniform in (in fact independent of)
$\hbar$. Next we discuss the propagation of quantities corresponding to moments
in the Wigner equation, that is, $xR^\hbar$  and
$\partial_yR^\hbar$. On the one hand, the equation for $xR^\hbar(t|\,x,y)$ is 
\begin{eqnarray*}
\partial_t(xR^\hbar(t|\,x,y)) &=& i\,\partial_x\partial_y(xR^\hbar(t|\,x,y))\,-\,i\,\frac{V(x+\frac12\hbar y)-V(x-\frac12\hbar y)}{\hbar}(xR^\hbar(t|\,x,y))
\\
&-& \,i\,\partial_yR^\hbar(t|\,x,y)\,,
\end{eqnarray*}
so that, arguing as above
\[
\|x\,R^\hbar(t)\|_{L^2}\,\le\,\|x\,R^\hbar(0)\|_{L^2}\,+\,\int_0^t\|\partial_yR^\hbar(s)\|_{L^2}\dD s\,.
\]
On the other hand, the equation for $\partial_yR^\hbar$ is
\begin{eqnarray*}
\partial_t\partial_yR^\hbar(t|\,x,y) &=& i\,\partial_x\partial_y\partial_yR^\hbar(t|\,x,y)\,-\,i\,\partial_y\left(\frac{V(x+\frac12\hbar y)-V(x-\frac12\hbar y)}{\hbar}\,R^\hbar(t|\,x,y)\right)
\\
&=& i\,\partial_x\partial_y\partial_y
    R^\hbar(t|\,x,y)\,-\,i\frac{V(x+\frac12\hbar y)-V(x-\frac12\hbar
    y)}{\hbar} \,\partial_yR^\hbar(t|\,x,y)
\\
&-&\frac i 2\,\left(V'\left(x+\frac12\hbar y\right)+V'\left(x-\frac12\hbar y\right)\right)\,R^\hbar(t|\,x,y)\,.
\end{eqnarray*}

Here it is worth mentioning that there are several options for the treatment of the last term on the last right-hand side. One could of course assume that $V'$ is bounded, but this is not very satisfying as it leaves out the case of a harmonic oscillator, where $V$ is quadratic in $x$.
In the latter case, one can assume instead that $V''$ is bounded. This
is the simpler case, allowing us to include a confining quadratic
potential if needed. Besides, since $V(x)\to+\infty$ as
$|x|\to+\infty$, it has a minimum point at $x_0\in\R$. Without
loss of generality, assume henceforth that $x_0=0$, so that $V'(0)=0$.
Therefore, by the mean value inequality
\[
|V'(z)-V'(0)|\,\le\,\|V''\|_{L^\infty}\,|z|\,,
\] 
and hence, having recast the last equality as
\begin{eqnarray*}
\partial_t\partial_y R^\hbar(t|\,x,y) &=& i\,\partial_x\partial_y\partial_yR^\hbar(t|\,x,y)\,-\,i\frac{V(x+\frac12\hbar y)\,-\,V(x-\frac12\hbar y)}{\hbar}\,\partial_yR^\hbar(t|\,x,y)
\\
&-&\,i\,V'(x)\,R^\hbar(t|\,x,y)\,-\,i\,\frac{V'(x+\tfrac12\hbar y)\,+\,V'(x-\tfrac12\hbar y)\,-\,2\,V'(x)}{2}\,R^\hbar(t|\,x,y)\,,
\end{eqnarray*}
we conclude as above that
$$
\|\partial_y R^\hbar(t)\|_{L^2}\,\le\,\|\partial_y R^\hbar(0)\|_{L^2}\,+\,\|V''\|_{L^\infty}\,\int_0^t\left(\|x\,R^\hbar(s)\|_{L^2}\,+\,\frac{\hbar}{2}\,\|y\,R^\hbar(s)\|_{L^2}\right)\,\dD s\,.
$$
Combining this inequality with the one for $x\,R^\hbar(t|\,x,y)$ leads to 
\begin{eqnarray*}
&&\|x\,R^\hbar(t)\|_{L^2}\,+\,\|\partial_yR^\hbar(t)\|_{L^2}\,-\,\|x\,R^\hbar(0)\|_{L^2}\,-\,\|\partial_yR^\hbar(0)\|_{L^2}
\\
&&\le\,\int_0^t\left( \|V''\|_{L^\infty}\|x\,R^\hbar(s)\|_{L^2}\,+\,\|\partial_yR^\hbar(s)\|_{L^2}\right)\,\dD s
\,+\,\frac{\hbar}{2}\, \|V''\|_{L^\infty}\, \int_0^t\|y\,R^\hbar(s)\|_{L^2}\,\dD s\,.
\end{eqnarray*}
The last term on the right-hand side is bounded as in \eqref{<1}, by
\[
\left(\|\partial_xR^\hbar(0)\|_{L^2}+\|y\,R^\hbar(0)\|_{L^2}\right)\,e^{t\max(1,\|V''\|_{L^\infty})}\,.
\]
Thus
\[
\begin{aligned}
& \|x\,R^\hbar(t)\|_{L^2} \,+\;\|\partial_yR^\hbar(t)\|_{L^2}\,-\,\|x\,R^\hbar(0)\|_{L^2}\,-\,\|\partial_yR^\hbar(0)\|_{L^2}
\\
&\le L\int_0^t(\|x\,R^\hbar(s\|_{L^2}\,+\,\|\partial_yR^\hbar(s)\|_{L^2})\,\dD s
\,+\,\frac{\hbar}{2}\,\left(\|\partial_xR^\hbar(0)\|_{L^2}\,+\,\|y\,R^\hbar(0)\|_{L^2}\right)\,(e^{Lt}-1)
\end{aligned}
\]
and by Gronwall's lemma, we conclude that
\begin{equation}\label{<2}
\begin{aligned}
\|x\,R^\hbar(t)\|_{L^2}\,+\,\|\partial_yR^\hbar(t)\|_{L^2} &\le\left(\|x\,R^\hbar(0)\|_{L^2}\,+\,\|\partial_yR^\hbar(0)\|_{L^2}\right)\,e^{Lt}
\\
& +\,\frac\hbar 2 \,\left(\|\partial_xR^\hbar(0)\|_{L^2}+\|y\,R^\hbar(0)\|_{L^2}\right)\,(e^{Lt}-1)\,e^{Lt}\,.
\end{aligned}
\end{equation}
Unlike \eqref{<1}, this bound is not independent of $\hbar$, but its $\hbar$-dependence is not singular as $\hbar\to 0$. Together with \eqref{<1}, the inequality \eqref{<2} is easily transformed into a uniform in $\hbar$ bound for $\hbar\in(0,2]$, say:
\begin{equation}\label{<3}
  \begin{aligned}
\cN_1(t) &\le\, \cN_1(0) \,e^{Lt} \,+\,\frac\hbar 2\,\left(\|\partial_xR^\hbar(0)\|_{L^2}+\|y\,R^\hbar(0)\|_{L^2}\right)\,(e^{Lt}-1)\,e^{Lt}
\\
&\le\, \cN_1(0) \,e^{Lt} \,+\,\frac\hbar 2\, \cN_1(0) \,(e^{Lt}-1)\,e^{Lt}
\\
&\le\, \cN_1(0)\,e^{2Lt}\,,
\end{aligned}
\end{equation}
which yields the desired result for $n=1$.

We now proceed in the same way, for each integer $a$, $b$, $\alpha$ and
$\beta$ such that $a+b+\alpha+\beta=n$ and seek an equation for 
\[
x^a\partial_x^by^\alpha\partial_y^\beta R^\hbar(t|\,x,y).
\]
Starting from
\[
\partial_tR^\hbar(t|\,x,y)\,=\,i\,\partial_x\partial_yR^\hbar(t|\,x,y)\,-\,i\,\frac{V(x+\frac12\hbar
  y)-V(x-\frac12\hbar y)}{\hbar} \, R^\hbar(t|\,x,y)\,,
\]
we see that
\[
\begin{aligned}
\partial_t\left(x^a\partial_x^by^\alpha\partial_y^\beta
  R^\hbar(t|\,x,y)\right) &\,=\,i\,\partial_x\partial_y\left(x^a\partial_x^by^\alpha\partial_y^\beta R^\hbar(t|\,x,y)\right)
\\
&\,-\,i\,\frac{V(x+\frac12\hbar y)-V(x-\frac12\hbar y)}{\hbar}\,\left(x^a\partial_x^by^\alpha\partial_y^\beta R^\hbar(t|\,x,y)\right)\,+\,i\,\cT_1 \,-\,i\,\cT_2\,,
\end{aligned}
\]
where
$$
\left\{
  \begin{array}{l}
\ds\cT_1 \,:=\, \left[x^a\partial_x^by^\alpha\partial_y^\beta,\,\partial_x\partial_y\right]\,R^\hbar(t|\,x,y)\,, 
    \\ [0.9em]
    \ds\cT_2 \,:=\, \left[x^a\partial_x^by^\alpha\partial_y^\beta,\frac{V(x+\frac12\hbar y)-V(x-\frac12\hbar y)}{\hbar}\right]\,R^\hbar(t|\,x,y)\,.
\end{array}\right.
    $$
Hence,  the idea is to treat $\cT_j$ as source terms for $j=1,2$, to find
\[
\begin{aligned}
\|x^a\partial_x^by^\alpha\partial_y^\beta R^\hbar(t|\,x,y)\|_{L^2}\,\le\,&\|x^a\partial_x^by^\alpha\partial_y^\beta R^\hbar(0|x,y)\|_{L^2}
\\
&\,+\,\int_0^t(\|\cT_1(s|\,\cdot,\cdot)\|_{L^2}\,+\,\|\cT_2(s|\,\cdot,\cdot)\|_{L^2})\dD s\,.
\end{aligned}
\]
First, we consider $\cT_1$ and use that
\[
\begin{aligned}
 \left[x^a\partial_x^by^\alpha\partial_y^\beta,\partial_x\partial_y\right]\, & =\, \left[x^a\partial_x^by^\alpha\partial_y^\beta,\partial_x\right]\,\partial_y\,+\,\partial_x\left[x^a\partial_x^by^\alpha\partial_y^\beta,\partial_y\right]\,,
\\
& \,=\, \left[x^a,\,\partial_x\right]\,\partial_x^by^\alpha\partial_y^\beta\partial_y\,+\,\partial_xx^a\partial_x^b\,\left[y^\alpha,\partial_y\right]\,\partial_y^\beta
\\
&\,=\,-ax^{a-1}\,\partial_x^by^\alpha\partial_y^{\beta+1}\,-\,\alpha\,\partial_xx^a\partial_x^by^{\alpha-1}\partial_y^\beta
\\
&\,=\,-a\,x^{a-1}\,\partial_x^by^\alpha\partial_y^{\beta+1}\,-\,\alpha
x^a\partial_x^{b+1}y^{\alpha-1}\partial_y^\beta \,-\,\alpha\,\left[\partial_x,x^a\right]\partial_x^by^{\alpha-1}\partial_y^\beta
\\
&\,=\,-a\,x^{a-1}\partial_x^by^\alpha\partial_y^{\beta+1}\,-\,\alpha x^a\partial_x^{b+1}y^{\alpha-1}\partial_y^\beta
\,-\,\alpha\, a\,x^{a-1}\,\partial_x^by^{\alpha-1}\partial_y^\beta\,.
\end{aligned}
\]
In the first two terms on the right-hand side, the sum of exponents is
\[
a-1+b+\alpha+\beta+1=a+b+1+\alpha-1+\beta=n\,,
\]
equal to that in $x^a\partial_x^by^\alpha\partial_y^\beta$, whereas the last term has a sum of exponents
\[
a-1+b+\alpha-1+\beta\,=\, n-2,
\]
less than that in $x^a\partial_x^by^\alpha\partial_y^\beta$. Hence we obtain
\[
\|\cT_1(s|\cdot,\cdot)\|_{L^2}\,\le\, 2\,n\,\cN_n(s)\,+\,n^2\cN_{n-2}(s)\,\le\,(2n+n^2)\,\cN_n(s)\,.
\]
Now for the second term $\cT_2$, we have
\[
\begin{aligned}
{}\left[x^a\partial_x^by^\alpha\partial_y^\beta,\frac{V(x+\frac\hbar2y)-V(x-\frac\hbar2y)}{\hbar}\right]
& =\,x^a\left[\partial_x^b,\frac{V(x+\frac\hbar2y)-V(x-\frac\hbar2y)}{\hbar}\right]y^\alpha\partial_y^\beta
\\
& +\,x^a\partial_x^by^\alpha\left[\partial_y^\beta,\frac{V(x+\frac\hbar2y)-V(x-\frac\hbar2y)}{\hbar}\right]\,.
\end{aligned}
\]
By Leibniz's formula
\[
\begin{aligned}
\left[\partial_x^b,\frac{V(x+\frac\hbar2y)-V(x-\frac\hbar2y)}{\hbar}\right]\,=\,&\sum_{j=1}^b{b\choose j}\,\frac{V^{(j)}(x+\frac\hbar2y)-V^{(j)}(x-\frac\hbar2y)}{\hbar}\,\partial_x^{b-j}\,,
\\
\left[\partial_y^\beta,\frac{V(x+\frac\hbar2y)-V(x-\frac\hbar2y)}{\hbar}\right]\,=\,&\sum_{k=1}^\beta{\beta\choose k}\,\frac{V^{(k)}(x+\frac\hbar2y)-(-1)^k\,V^{(k)}(x-\frac\hbar2y)}{2^k\,\hbar}\,\hbar^k\,\partial_y^{\beta-k}\,,
\end{aligned}
\]
so that
\[
\begin{aligned}
& \left[x^a\partial_x^by^\alpha\partial_y^\beta,\frac{V(x+\frac\hbar2y)-V(x-\frac\hbar2y)}{\hbar}\right]
\\
&\quad \,=\,\sum_{j=1}^b{b\choose j}\frac{V^{(j)}(x+\frac\hbar2y)-V^{(j)}(x-\frac\hbar2y)}{\hbar}x^a\partial_x^{b-j}y^\alpha\partial_y^\beta
\\
&\quad \quad\,+\,\sum_{k=1}^\beta{\beta\choose k}\frac{V^{(k)}(x+\frac\hbar2y)-(-1)^k\,V^{(k)}(x-\frac\hbar2y)}{2^k\hbar}\hbar^kx^a\partial_x^by^\alpha\partial_y^{\beta-k}
\\
&\quad \quad\,+\,\sum_{k=1}^\beta{\beta\choose k}x^a\left[\partial^b_x,\frac{V^{(k)}(x+\frac\hbar2y)-(-1)^k\,V^{(k)}(x-\frac\hbar2y)}{2^k\hbar}\right]\hbar^ky^\alpha\partial_y^{\beta-k}
\\
&\quad\,=\,b\,\frac{V'(x+\frac\hbar2y)-V'(x-\frac\hbar2y)}{\hbar y}\,x^a\partial_x^{b-1}y^{\alpha+1}\partial_y^\beta
\\
&\quad \quad\,+\,\sum_{j=2}^b{b\choose j}\, \frac{V^{(j)}(x+\frac\hbar2y)-V^{(j)}(x-\frac\hbar2y)}{\hbar y}\,x^a\,\partial_x^{b-j}y^{\alpha+1}\partial_y^\beta
\\
&\quad \quad\,+\,\beta\,\frac{V'(x+\frac\hbar2y)+V'(x-\frac\hbar2y)-2\,V'(x)}{2\hbar y}\,\hbar \,x^a\,\partial_x^by^{\alpha+1}\partial_y^{\beta-1}
\\
&\quad \quad\,+\,\beta\,\frac{V'(x)-V'(0)}{x}\,x^{a+1}\,\partial_x^by^\alpha\partial_y^{\beta-1}\,+\,\beta\, V'(0)\,x^a\,\partial_x^by^\alpha\partial_y^{\beta-1}
\\
&\quad \quad\,+\,\sum_{k=2}^\beta{\beta\choose k}\,\frac{V^{(k)}(x+\frac\hbar2y)-(-1)^k\,V^{(k)}(x-\frac\hbar2y)}{2^k\hbar}\,\hbar^k\,x^a\,\partial_x^by^\alpha\partial_y^{\beta-k}
\\
&\quad \quad\,+\,\sum_{k=1}^\beta\sum_{j=1}^b{\beta\choose k}{b\choose j}\,\frac{V^{(k+j)}(x+\frac\hbar2y)-(-1)^k\,V^{(k+j)}(x-\frac\hbar2y)}{2^k\hbar}\,\hbar^k\,x^a\,\partial^{b-j}_xy^\alpha\partial_y^{\beta-k}\,.
\end{aligned}
\]
The first term in the right-hand side involves the ratio
\[
\left|\,\frac{V'(x+\frac\hbar2y)\,-\,V'(x-\frac\hbar2y)}{\hbar\; y}\right|\,\le\|V''\|_{L^\infty}\,,
\]
and the monomial $x^a\partial_x^{b-1}y^{\alpha+1}\partial_y^\beta$, whose exponents add up to
\[
a+b-1+\alpha+1+\beta=n\,,
\]
which is equal to the sum of exponents in the monomial
$x^a\partial_x^by^\alpha\partial_y^\beta$. The second term in the
right-hand side involves monomials whose exponents add up to
\[
a+b-j+\alpha+1+\beta\,<\,n
\] 
since $j\ge 2$, multiplied by the ratio
\[
\left|\,\frac{V^{(j)}(x+\frac\hbar2y)\,-\,V^{(j)}(x-\frac\hbar2y)}{\hbar\, y}\right|\,\le\|V^{(j+1)}\|_{L^\infty}<\infty\,.
\]
The last three terms on the right-hand side involve monomials whose exponents add up to 
\[
a+b-j+\alpha+\beta-k\,<\,n
\]
with $j+k\ge 1$, multiplied by $ V'(0)$ or  a ratio of the form
\[
\left|\,\frac{V^{(k+j)}(x+\frac\hbar2y)-(-1)^k\,V^{(k+j)}(x-\frac\hbar2y)}{2^k\hbar}\right|\,\hbar^k\,\le\,\|V^{(k+j)}\|_{L^\infty}\,\hbar^{k-1}
\]
with $k\ge 1$ and $k+j\ge 2$. There remains the third and fourth terms, involving monomials whose exponents add up to $a+b+\alpha+\beta=n$, multiplied by ratios of the form
\[
\left|\,\frac{V'(x+\frac\hbar2y)\,+\,V'(x-\frac\hbar2y)-2\,V'(x)}{2\hbar\, y}\right|\,\le\,\frac12\|V''\|_{L^\infty}\,<\,\infty\,,
\]
or
\[
\left|\,\frac{V'(x)-V'(0)}{x}\right|\,\le\|V''\|_{L^\infty}<\infty\,.
\]
Summarizing, assuming that $\hbar\le 1$ while $a+b+\alpha+\beta=n$ with $a,b,\alpha,\beta\ge 0$,
\[
\begin{aligned}
\|\cT_2(s|\,\cdot,\cdot)\|_{L^2} &\le\, n\,\|V''\|_{L^\infty}\,\cN_n(s)\,+\,\underbrace{\sum_{j=2}^b{b\choose j}\|V^{(j+1)}\|_{L^\infty}\,\cN_{n-j+1}(s)}_{\le 2^n\max_{2\le i\le n}\|V^{(i+1)}\|_{L^\infty}\cN_{n-1}(s)}\,+\,\frac12\,\|V''\|_{L^\infty}\,n\,\hbar\,\cN_n(s)
\\
&\quad\,+\,\|V''\|_{L^\infty}\,n\,\cN_n(s)\,+\,n\,|V'(0)|\,\cN_{n-1}(s)\,+\,\underbrace{\sum_{k=2}^\beta{\beta\choose k}\|V^{(k)}\|_{L^\infty}\,\hbar^{k-1}\,\cN_{n-k}(s)}_{\le\, 2^n\,\max_{2\le i\le n}\|V^{(i)}\|_{L^\infty}\,\cN_{n-2}(s)}
\\
&\quad\,+\,\underbrace{\sum_{k=1}^\beta\sum_{j=1}^b{\beta\choose
    k}{b\choose j}\|V^{(k+j)}\|_{L^\infty}\,\hbar^{k-1}\,\cN_{n-k-j}(s)}_{\le
  \,2^n\cdot 2^n\,\max_{2\le i\le n}\|V^{(i)}\|_{L^\infty}\cN_{n-2}(s)}\,.
\end{aligned}
\]
Hence, it yields
\[
\begin{aligned}
\|\cT_2(s|\cdot,\cdot)\|_{L^2} &\le\,\max\left(|V'(0),\,\max_{2\le
    k\le
    n+1}\|V^{(k)}\|_{L^\infty}\right)\,\left(n+2^n+\frac12\,n\,\hbar+n+n+2^n\hbar+2^n\cdot
  2^n\right)\,\cN_n(s)
\\
&\le\,\max\left(|V'(0),\,\max_{2\le k\le n+1}\|V^{(k)}\|_{L^\infty}\right)\left(\frac72+2^n\right)\,(2^n+2)\,\cN_n(s)
\end{aligned}
\]
since $0<\hbar\le 1$.
\\

Thus, we have proved that
\begin{eqnarray*}
 && \|x^a\partial_x^by^\alpha\partial_y^\beta
R^\hbar(t)\|_{L^2} \,\le\, \|x^a\partial_x^by^\alpha\partial_y^\beta
R^\hbar(0)\|_{L^2}
\\[0.85em]
&&+\,\left(1+\max\left(|V'(0),\max_{2\le k\le n+1}\|V^{(k)}\|_{L^\infty}\right)\right)\,\int_0^t\left(\frac72\,+\,2^n\right)\,(2^n+2)\,\cN_n(s)\,\dD s\,,
\end{eqnarray*}
so that
\begin{eqnarray*}
&& \cN_n(t)\,\le\,\cN_n(0)
\\[0.85em]
&&+\left(1+\max\left(|V'(0),\max_{2\le k\le n+1}\|V^{(k)}\|_{L^\infty}\right)\right)\,\left(\frac72+2^n\right)\,(2^n+2)(n+1)^4\,\int_0^t\cN_n(s)\dD s
\end{eqnarray*}
and one concludes by Gronwall's inequality.

%
%
%
%

\bibliographystyle{abbrv}
\bibliography{refer}

\end{document}